\documentclass[12pt,a4paper,oneside]{amsart}

\usepackage{amsmath}
\usepackage{amsthm}
\usepackage{amsfonts}
\usepackage{amssymb}
\usepackage{amscd}

\usepackage[T1]{fontenc}
\usepackage{textcomp}

\usepackage[warnunknown]{ucs}
\usepackage[utf8x]{inputenc}

\usepackage{mathpazo}
\usepackage[scaled]{helvet}

\usepackage{calc}

\usepackage[pdftex,final]{graphicx}
\usepackage[usenames]{color}

\usepackage{float}

\usepackage{microtype}

\usepackage{url}

\usepackage[verbose=true,heightrounded=true]{geometry}

\usepackage{booktabs}

\usepackage{paralist}

\usepackage{varwidth}
\newcommand{\cellboxt}[1]{\begin{varwidth}[t]{0.5\textwidth}#1\end{varwidth}}
\newcommand{\xquad}{\hspace*{1em}\relax}
\newcommand{\xqquad}{\hspace*{2em}\relax}
\newcommand{\xqqquad}{\hspace*{3em}\relax}

\usepackage{xfrac}

\theoremstyle{theorem}
\newtheorem{theorem}{Theorem}[section]
\newtheorem{proposition}[theorem]{Proposition}

\numberwithin{equation}{section}

\begin{document}
\title
[Mathematical classification of square tiles]{A contribution for a mathematical \\
classification of square tiles}
\author{Jorge Rezende}
\address{Grupo de F\'{i}sica-Matem\'{a}tica da Universidade de Lisboa,
Av. Prof. Gama Pinto 2, 1649-003 Lisboa, Portugal,
and Departamento de Matem\'{a}tica,
Faculdade de Ci\^{e}ncias da Universidade de Lisboa}
\email{rezende@cii.fc.ul.pt}
\thanks
{The Mathematical Physics Group is supported by the (Portuguese) 
Foundation for Science and Technology (FCT)}
\thanks
{This paper is in final form and no version of it will be submitted for
publication elsewhere.}
\maketitle

\section{Introduction}

\noindent
In this article we shall study some geometric properties of a non-trivial
square tile (a non-trivial square tile is a non-constant function on a
square).

Consider infinitely many copies of this single square tile and cover the
plane with them, without gaps and without overlaps (a tiling of the plane),
with the vertices making a square point lattice. The question we ask
ourselves in this article is the following: if there is a rotation center of
order $4$ what kind of geometric properties has the drawing in the tile?

\subsection{Trivial patterns}

In this article we represent rotation centers of order $4$ by small black
squares; those of order $2$ by small circles; reflections are represented by
red lines. We do not represent glide reflections. The border of fundamental
regions are represented by yellow lines.

\begin{figure}[h]
  \centering
  \includegraphics[width=3.3131in]{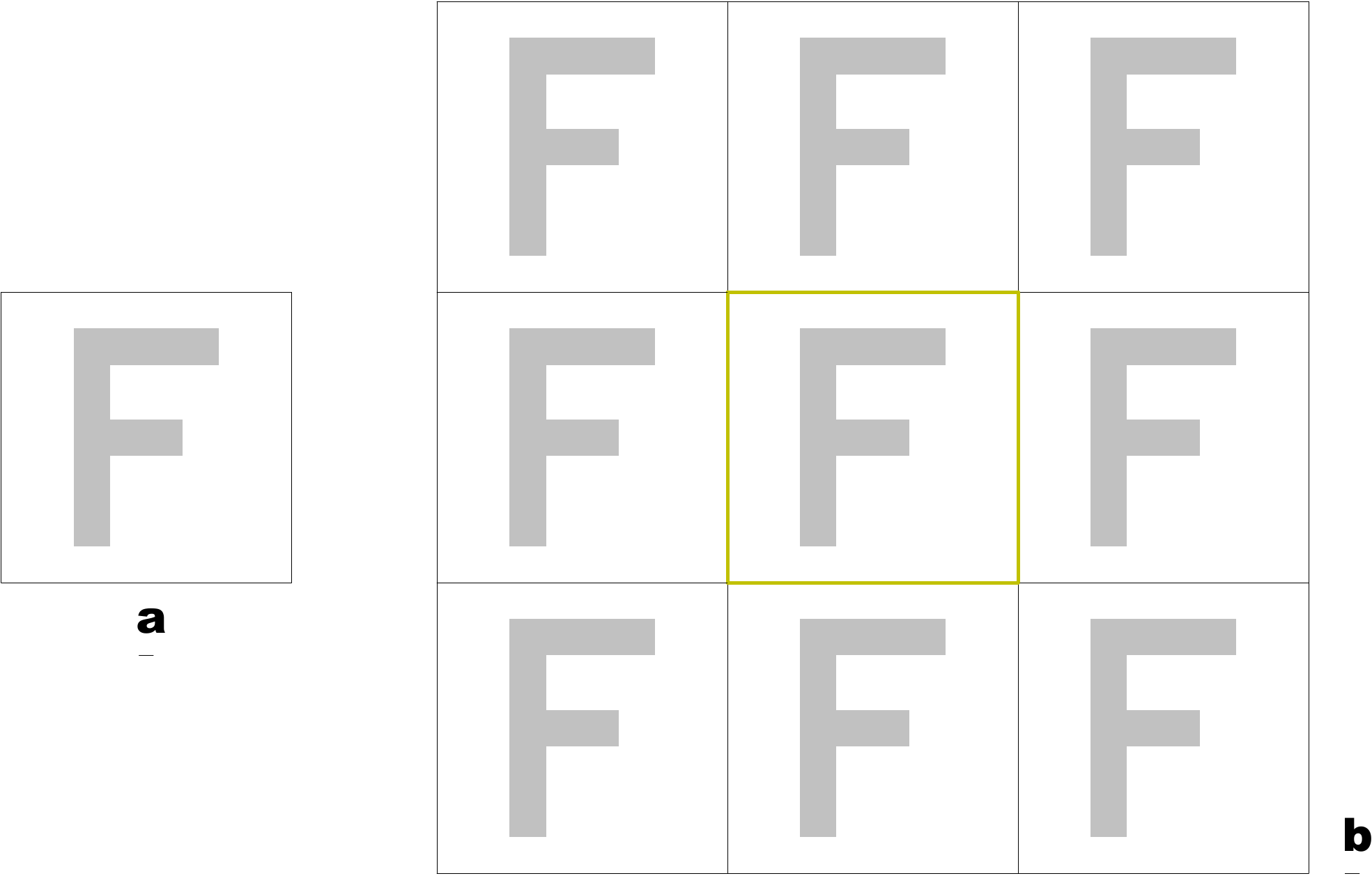}  
  \caption{}
  \label{fig:01}
\end{figure}

\begin{figure}[h]
  \centering
  \includegraphics[width=3.6789in]{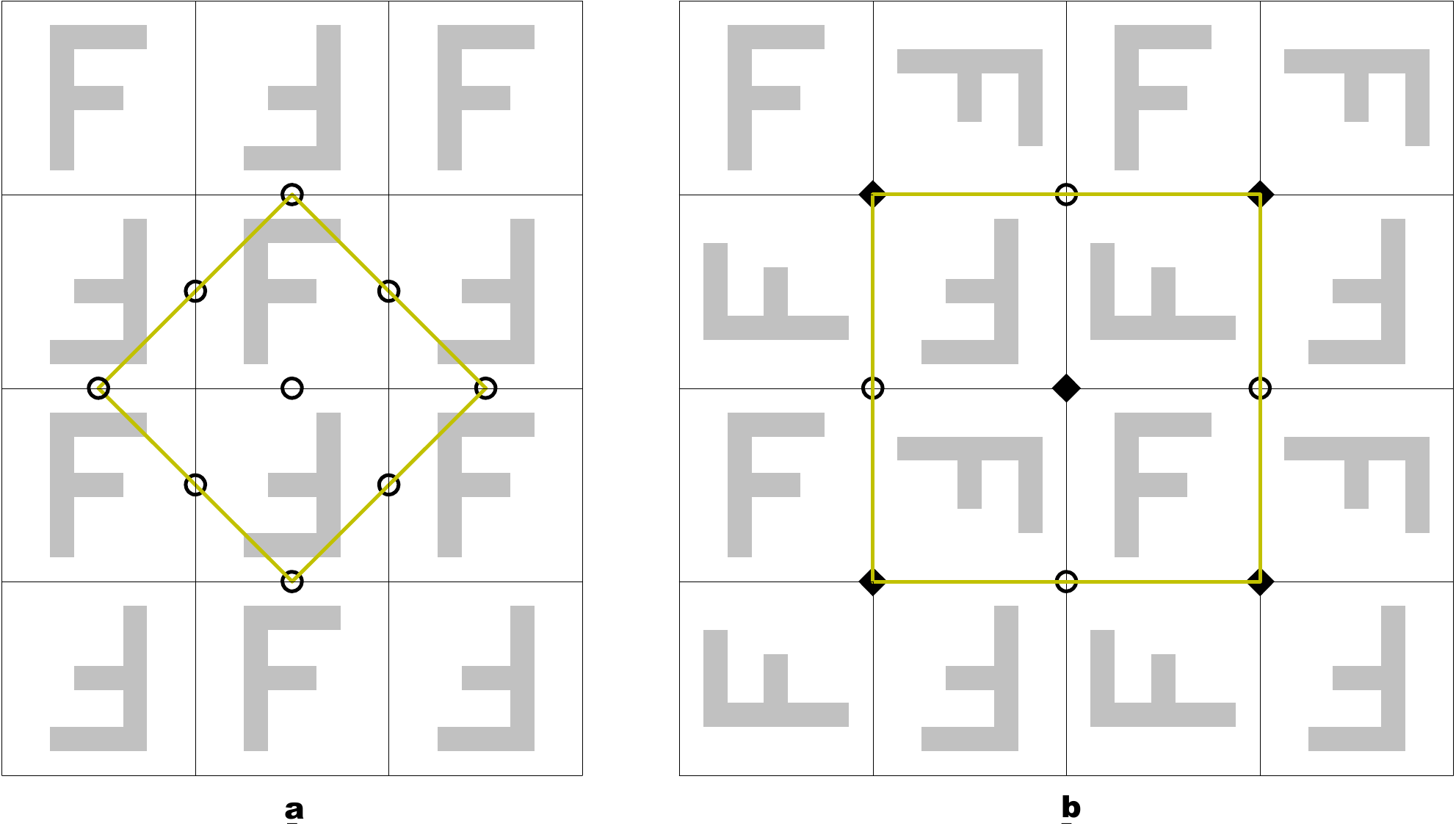}  
  \caption{}
  \label{fig:02}
\end{figure}

Take a square tile like the one that is shown in Figure \ref{fig:01}a. If one
translates it (Figure \ref{fig:01}b), rotates it around a vertex 
(as in Figure \ref{fig:02}b) or around the center of one edge 
(as in Figure \ref{fig:02}a), one can obtain different
patterns. These are the trivial ones and these operations (translations,
rotations around a vertex or, of order $2$, around the center of one edge)
are also called trivial.

In the trivial patterns the rotation centers of order $4$ can only be in the
vertices, and the rotation centers of order $2$ can only be in the vertices
or in the center of the edges; these are the trivial centers.

All compositions with a 1968 tile by Querubim Lapa (see page 31, in
\cite{Almeida}) have rotation centers of order $2$ and $4$ that are
obtained in this way.

However, a 1966 Eduardo Nery tile 
(Figures \ref{fig:03} and \ref{fig:28}--\ref{fig:30}) 
can help us to see
that there are tilings of the plane with infinitely many copies of a single
and ``simple'' tile that are not trivial.

\begin{figure}[h]
  \centering
  \includegraphics[width=4.139in]{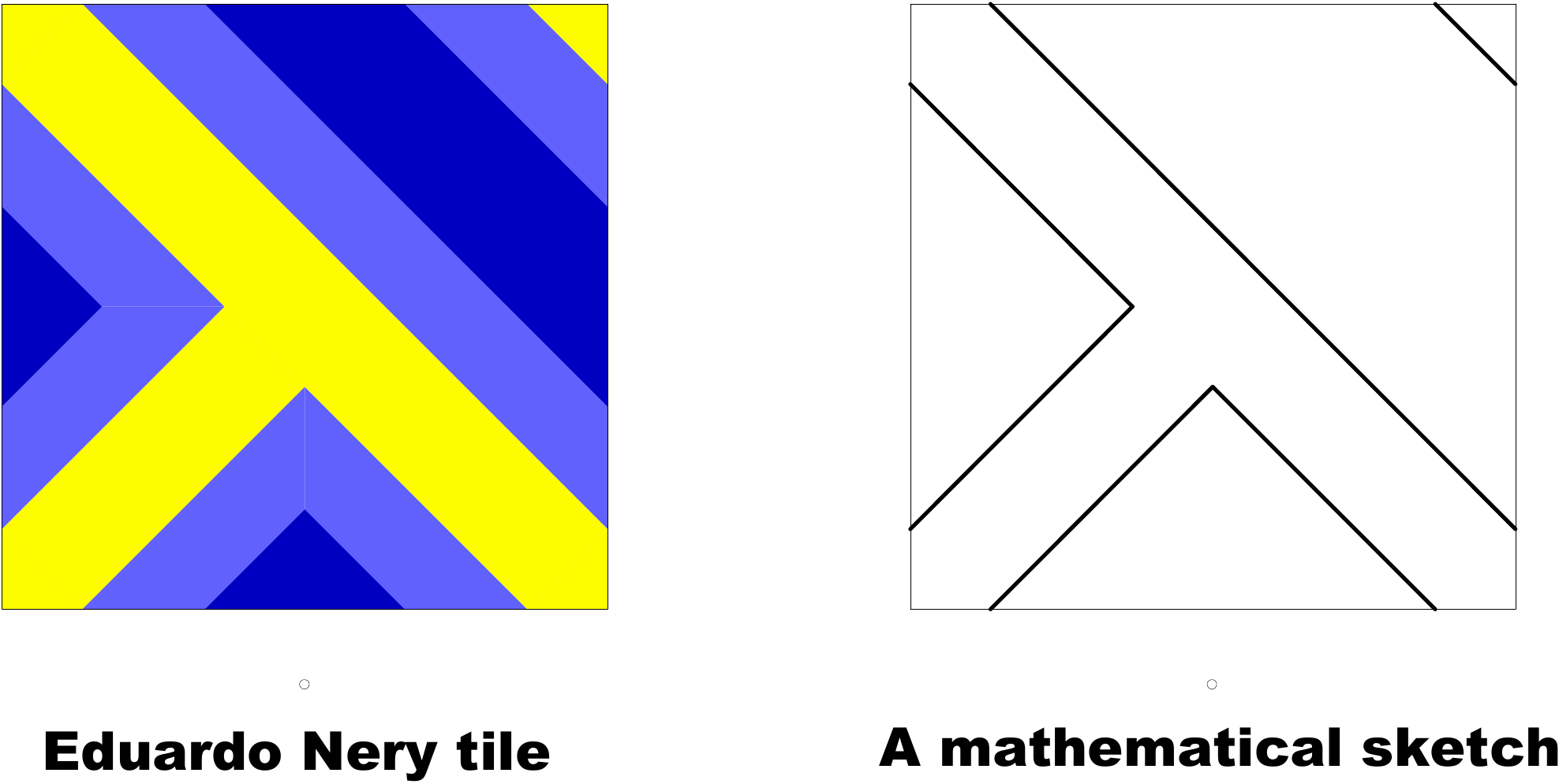}
  \caption{}
  \label{fig:03}
\end{figure}

\subsection{A Eduardo Nery tile}

In 1966, the Portuguese artist Eduardo Nery designed a pattern tile,
with remarkable properties (see Figures \ref{fig:03} and
\ref{fig:28}--\ref{fig:30}).  It can be seen in several locations in
Portugal (see References \cite{Almeida}, \cite{Nery1} and \cite
{Nery2}) such as:

\begin{compactenum}[a)]
\item the agency of the former Nacional Ultramarino Bank, Torres
  Vedras (1971--1972);
\item the M\'{e}rtola Health Centre (1981);
\item the Contumil Railway Station, Oporto (65 panels in three passenger
platforms, 1992--1994).
\end{compactenum}

\begin{figure}[b]
  \centering
  \includegraphics[width=4.6622in]{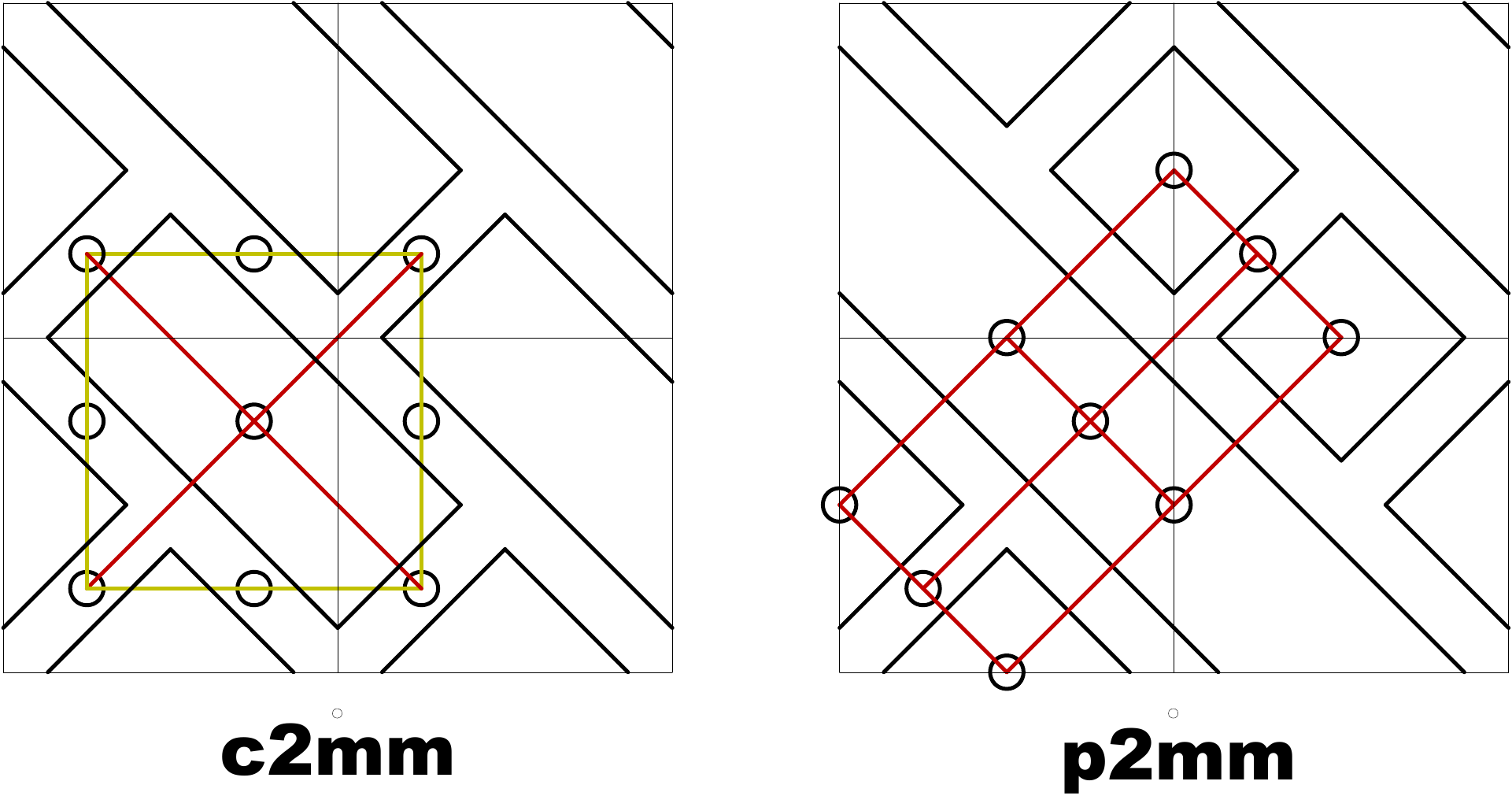}
  \vspace*{-18pt}
  \caption{}
  \label{fig:04}
\end{figure}
\begin{figure}[b]
  \centering
  \vspace*{2pt}
  \includegraphics[width=3.8372in]{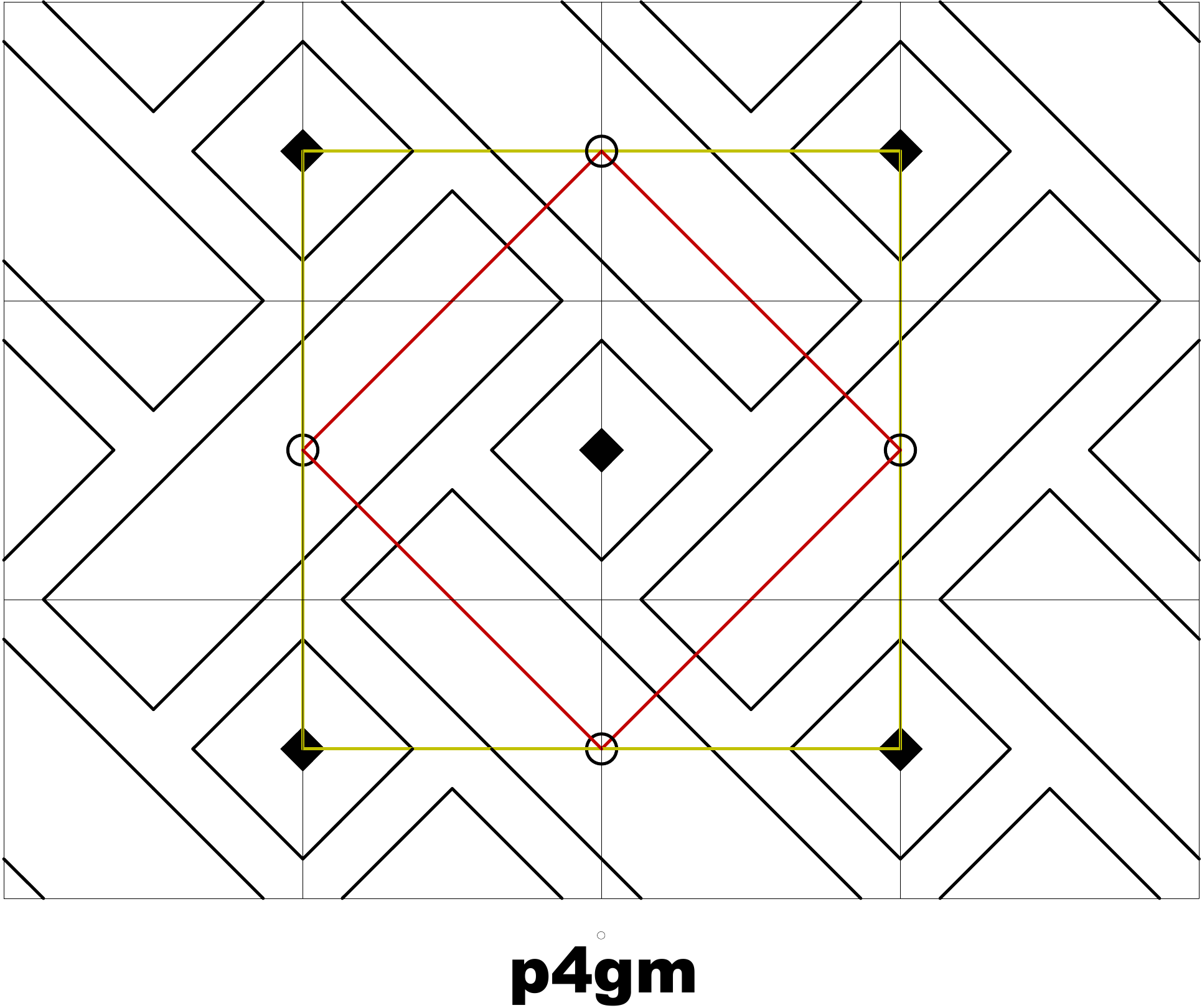}  
  \vspace*{-8pt}
  \caption{}
  \label{fig:05}
\end{figure}

If one assembles four of these tiles around a corner making a square and
then translates it parallelly to the edges, covering the plane, one obtains
eleven ($11$) different tilings with symmetry by reflexion and twelve 
$\left( 12=6\times 2\right) $ without it. In the eleven ($11$) tilings are
represented eight ($8$) wallpaper groups: $pm$ ($1$), $p2mm$ ($1$), $p2mg$ 
($1$), $p2gg$ ($1$), $cm$ ($3$), $c2mm$ ($1$), $p4mm$ ($1$) and $p4gm$ ($2$).
In the twelve ($12$) tilings without symmetry by reflexion are represented
two wallpaper groups: $p1$ ($10$) and $p2$ ($2$). The meaning of these
notations on wallpaper groups can be seen in References \cite{Armstrong}, 
\cite{rezende1}, \cite{Schattschneider}.

Three of these tilings are represented in Figures \ref{fig:04} and \ref{fig:05}.
They show
rotations centers of order $2$ in the interior of the tiles and rotation
centers of order $4$ in the middle of some edges. This means that they are
not trivial. In the Appendix, Figures \ref{fig:28}--\ref{fig:30} 
show the same patterns in coloured tilings.

Figure \ref{fig:06} represents all the symmetries of this Eduardo Nery tile.
\vspace*{-6pt}
\begin{figure}[H]
  \centering
  \includegraphics[width=2.5036in]{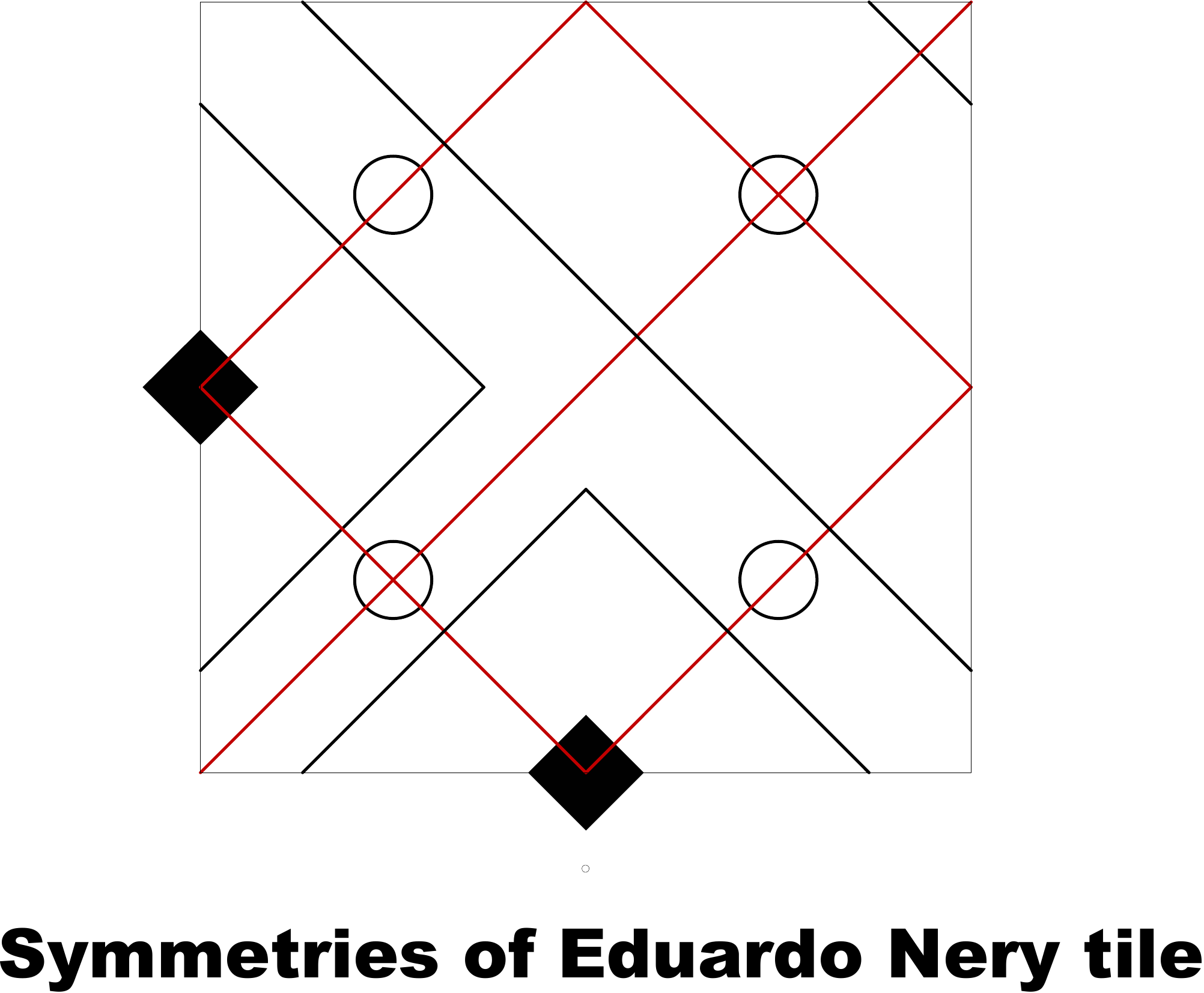}
  \vspace*{-12pt}
  \caption{}
  \label{fig:06}
\end{figure}

These interesting three examples of non trivial tilings of the plane 
(in particular, the one of Figures \ref{fig:05} and \ref{fig:30})
encouraged this search for an
answer to the question that we asked ourselves in the beginning: if there is
a rotation center of order $4$ what kind of geometric properties has the
drawing in the tile?

\section{The general rules and their exceptions}

In the following, remember that we represent rotation centers of order $4$
by small black squares; those of order $2$ by small circles; reflections are
represented by red lines. We do not represent glide reflections. The border
of fundamental regions are represented by yellow lines.

Consider a $p4$ pattern on the plane ($\mathbb{R}^2$, the referential of the
pattern) with rotation centers of order $4$ located in the points $\left(
r,s\right) \in \mathbb{Z}^2$ ($r$ and $s$ are integer numbers).

Let $\left( p,q\right) ,\left( r,s\right) \in \mathbb{Z}^2$. Consider the
coordinate system defined by the origin $\left( 0,0\right) $ and the vectors 
$u=\left( p,q\right) $, $v=\left( -q,p\right) $. Then
\begin{align*}
\left( 1,0\right) &=\frac 1{p^2+q^2}\left( pu-qv\right) \text{,}\displaybreak[0]\\
\left( 0,1\right) &=\frac 1{p^2+q^2}\left( qu+pv\right) \text{.}
\end{align*}

Hence, in the referential with $\left( 0,0\right) $ as origin and $u$ and $v$
as vectors (the referential of the tile; see Figure \ref{fig:07}), 
the point $\left(
r,s\right) \in \mathbb{Z}^2$, in the referential of the pattern, has
coordinates 
\begin{equation*}
\left( r,s\right) \rightarrow \frac 1{p^2+q^2}\left( rp+sq,sp-rq\right) 
\text{.}
\end{equation*}
\begin{figure}[H]
  \centering
  \includegraphics[width=2.4855in]{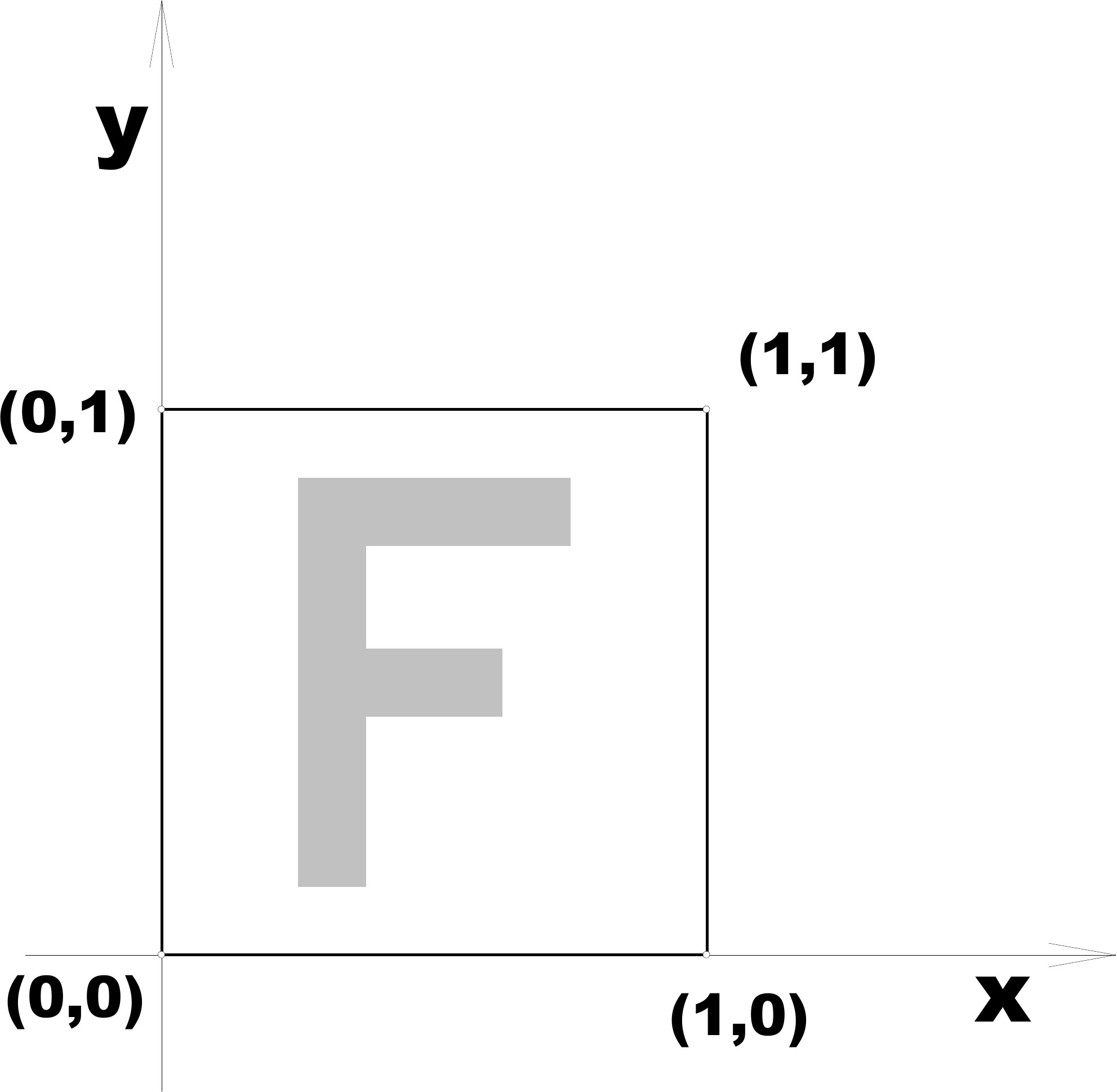}  
  \caption{}
  \label{fig:07}
\end{figure}


\textbf{In the following we shall always represent the domain of a tile as
the square }$\left[ 0,1\right]^2$ (see Figure \ref{fig:07}). 
In a tiling of the plane,
with infinitely many copies of this square tile, each tile is located in a
region 
$\left[ p,p+1\right] \times \left[ q,q+1\right] $, with $p,q\in \mathbb{Z}$. 
If no confusion is possible we use the domain to denote each tile.

The majority of the tiles that one can see are of one of \textbf{four types}. 
However there are \textbf{eight exceptions} to these general rules.

\subsection{The first type of general tiles}

Take $p$ and $q$ both even or both odd numbers ($p+q$ is even) and the
square with vertices in $\left( 0,0\right) $, $\left( p,q\right) $, $\left(
-q,p\right) $, $\left( p-q,p+q\right) $, in the referential of the pattern,
or \textit{any translation} of this square. A tile of this type is obtained
by intersecting this square with the pattern in the plane.

The rotation centers of order $4$ are located, in the referential of the
tile, at 
\begin{equation*}
\left( a_0,b_0\right) +\frac 1{p^2+q^2}\left( rp+sq,sp-rq\right) \text{,}
\end{equation*}
for some $\left( a_0,b_0\right) \in \left[ 0,1\right]^2$, and with $r,s\in 
\mathbb{Z}$.

Define $p_1=\frac{p-q}2,q_1=\frac{p+q}2$; then $p=p_1+q_1$ and $q=q_1-p_1$.
Notice that for any $p_1,q_1\in \mathbb{Z}$, $p_1+q_1$ and $q_1-p_1$ are both
even or both odd numbers. Then, the rotation centers of order $4$ are
located, in the referential of the tile, at 
\begin{equation*}
\left( a_0,b_0\right) +\frac 1{2\left( p_1^2+q_1^2\right) }\left( r\left(
p_1+q_1\right) +s\left( q_1-p_1\right) ,s\left( p_1+q_1\right) +r\left(
p_1-q_1\right) \right) \text{,}
\end{equation*}
for some $\left( a_0,b_0\right) \in \left[ 0,1\right]^2$, and with $r,s\in 
\mathbb{Z}$.

The number of such rotation centers is an even number, $p^2+q^2=2\left(
p_1^2+q_1^2\right) $.

\begin{figure}[h]
  \centering
  \includegraphics[width=4.0136in]{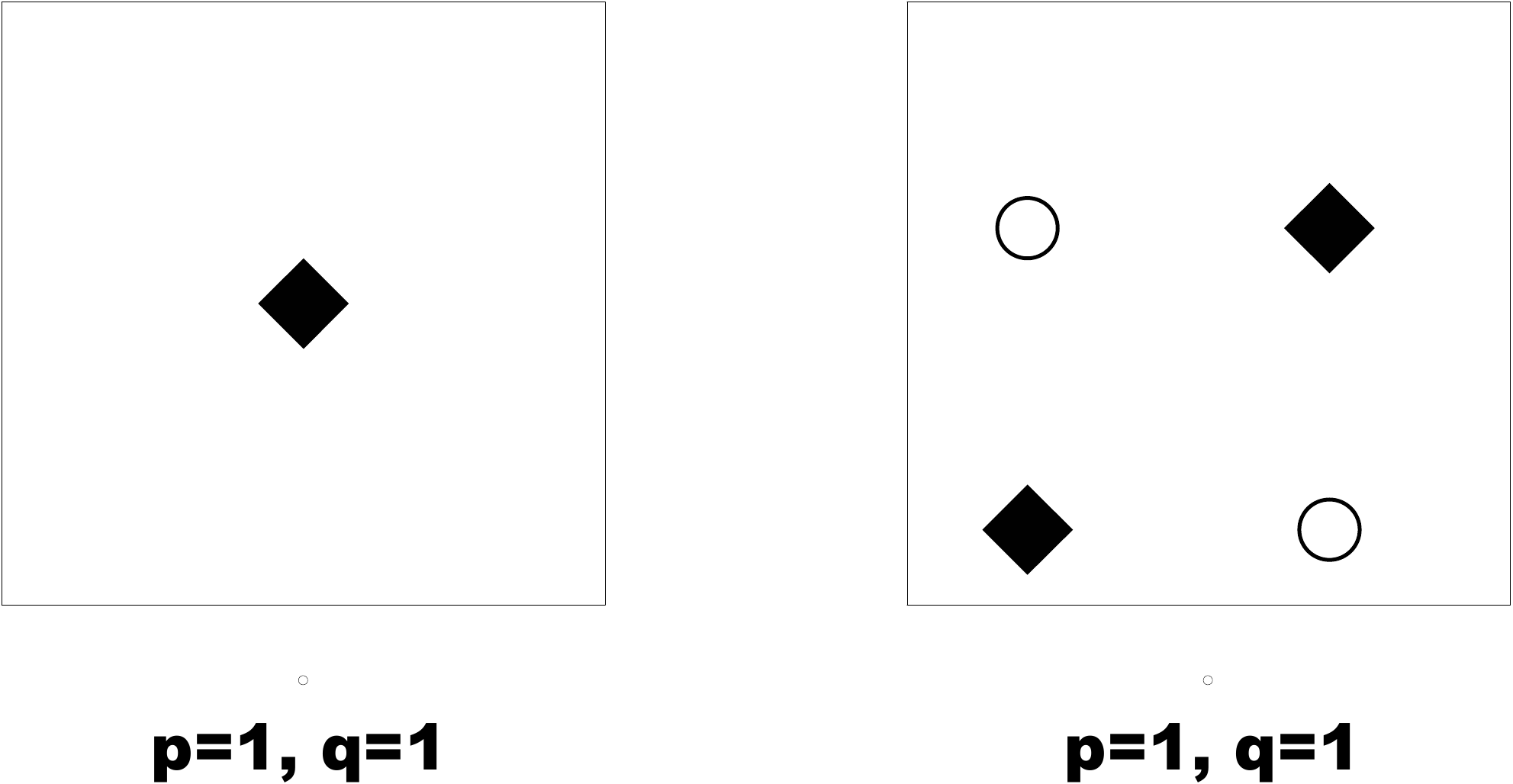}  
  \caption{}
  \label{fig:08}
\end{figure}

In Figure \ref{fig:08} we consider $p=q=1$ and represent two tiles. 
In the left one
there is a center at $\left( \sfrac 12,\sfrac 12\right) $ and the other
centers are trivial, located in the vertices. In the right one there is a
center at $\left( a_0,b_0\right) $ and the other at 
$\left( a_0+\sfrac 12,b_0+\sfrac 12\right) $. We only represent $p4$ versions of the tiles;
there are also $p4mm$ versions and $p4gm$ versions.

\begin{figure}[h]
  \centering
  \includegraphics[width=4.139in]{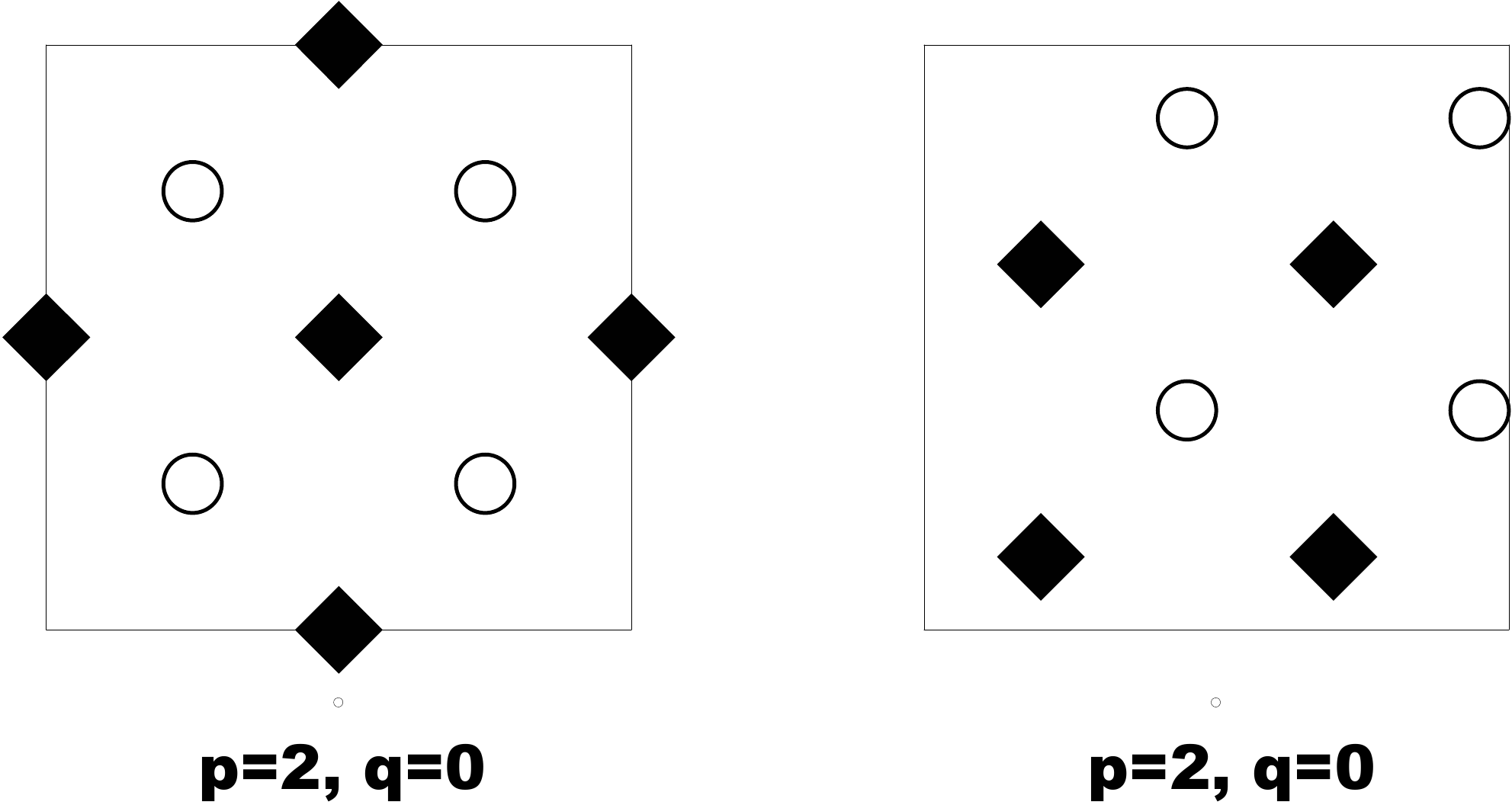}  
  \caption{}
  \label{fig:09}
\end{figure}

In Figure \ref{fig:09} we consider $p=2$ and $q=0$ and represent two tiles. 
In the left
one there are centers at $\left( \sfrac 12,0\right) $, $\left( 0,\sfrac
12\right) $, $\left( \sfrac 12,\sfrac 12\right) $, $\left( 1,\sfrac 12\right)$,
$\left( \sfrac 12,1\right) $ and the others are located in the vertices. In
the right one there is a center at $\left( a_0,b_0\right) $ and the others
at $\left( a_0+\sfrac 12,b_0\right) $, 
$\left( a_0+\sfrac 12,b_0+\sfrac 12\right) $ and $\left( a_0,b_0+\sfrac 12\right) $.
We only represent $p4$ versions of the tiles; there are also $p4mm$
versions and $p4gm$ versions.

\subsection{The second type of general tiles}

Take $p$ and $q$ such that if one of them is even the other is odd ($p+q$ is
odd) and take the square with vertices in $\left( 0,0\right) $, $\left(
p,q\right) $, $\left( -q,p\right) $, $\left( p-q,p+q\right) $, in the
referential of the pattern. A tile of this second type is obtained by
intersecting this square with the pattern in the plane.

The rotation centers of order $4$ are located, in the referential of the
tile, at 
\begin{equation*}
\frac 1{p^2+q^2}\left( rp+sq,sp-rq\right) \text{,}
\end{equation*}
with $r,s\in \mathbb{Z}$.

The number of such rotation centers is a odd number, $p^2+q^2$.

\begin{figure}[h]
  \centering
  \includegraphics[width=4.0136in]{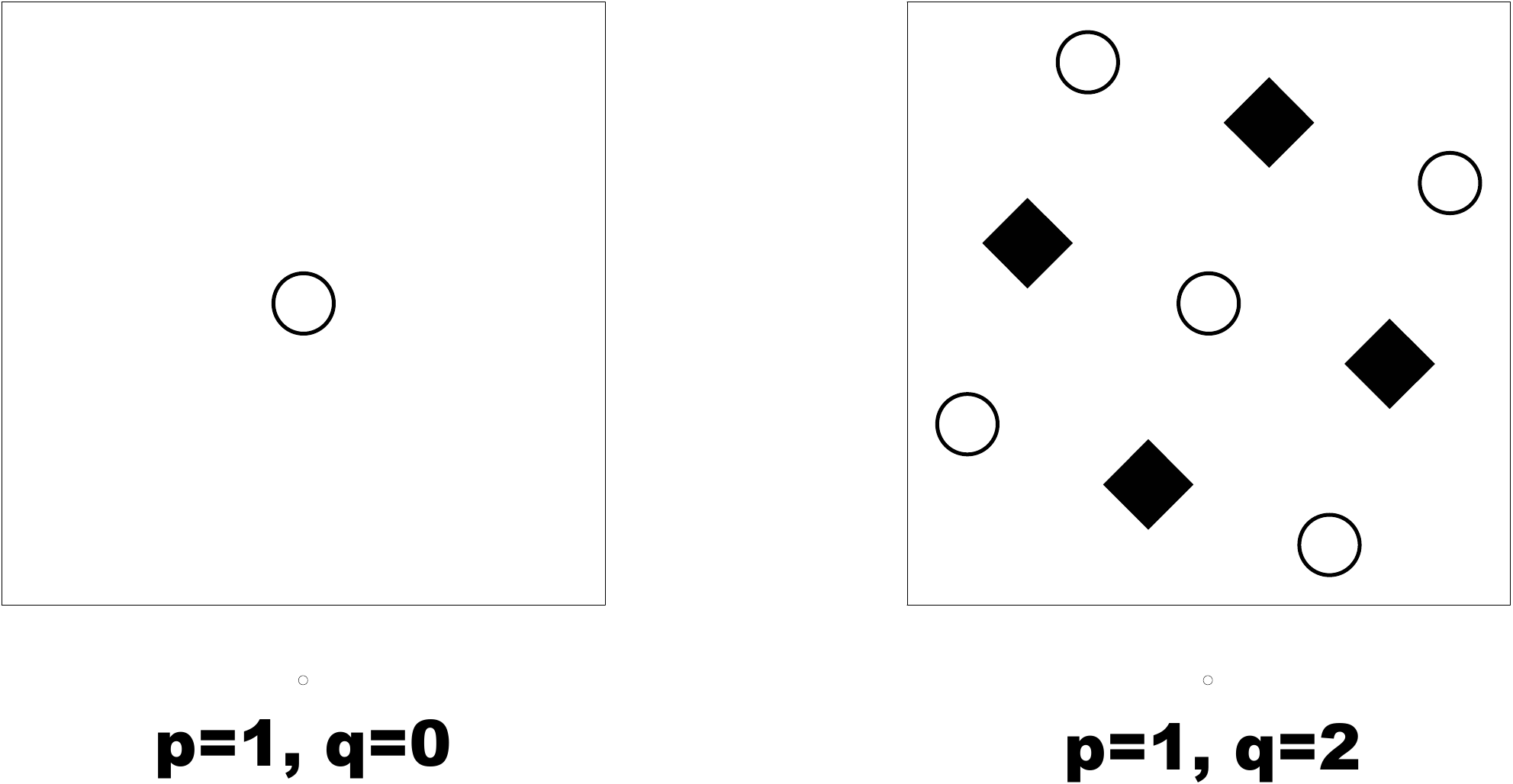}  
  \caption{}
  \label{fig:10}
\end{figure}

In Figure \ref{fig:10} we represent two tiles. 
In the left we consider $p=1$ and $q=0$; 
there are only centers located in the vertices. In the right we consider 
$p=1$ and $q=2$; there are centers at $\left( \sfrac 25,\sfrac 15\right) $, 
$\left( \sfrac 45,\sfrac 25\right) $, $\left( \sfrac 35,\sfrac 45\right) $, 
$\left( \sfrac 15,\sfrac 35\right) $, and the others are located in the
vertices. We only represent $p4$ versions of the tiles; there are also $p4mm$
versions and $p4gm$ versions. We can also consider interchanging $p$ and $q$; 
the figures are the same but seen as in a mirror.

\subsection{The third type of general tiles}

Take $p$ and $q$ such that if one of them is even the other is odd ($p+q$ is
odd) and take the square with vertices in $\left( -\sfrac p2,-\sfrac q2\right) 
$, $\left( \sfrac p2,\sfrac q2\right) $, $\left( \sfrac p2-q,p+\sfrac q2\right) $,
$\left( -\sfrac p2-q,p-\sfrac q2\right) $, in the referential of the
pattern. A tile of this third type is obtained by intersecting this square
with the pattern in the plane.

The rotation centers of order $4$ are located, in the referential of the
tile, at 
\begin{equation*}
\left( a,b\right) +\frac 1{p^2+q^2}\left( rp+sq,sp-rq\right) \text{,}
\end{equation*}
with $r,s\in \mathbb{Z}$, $\left( a,b\right) =\left( \frac 12,0\right) $ or 
$\left( a,b\right) =\left( 0,\frac 12\right) $.

The number of such rotation centers is a odd number, $p^2+q^2$.

\begin{figure}[h]
  \centering
  \includegraphics[width=4.0136in]{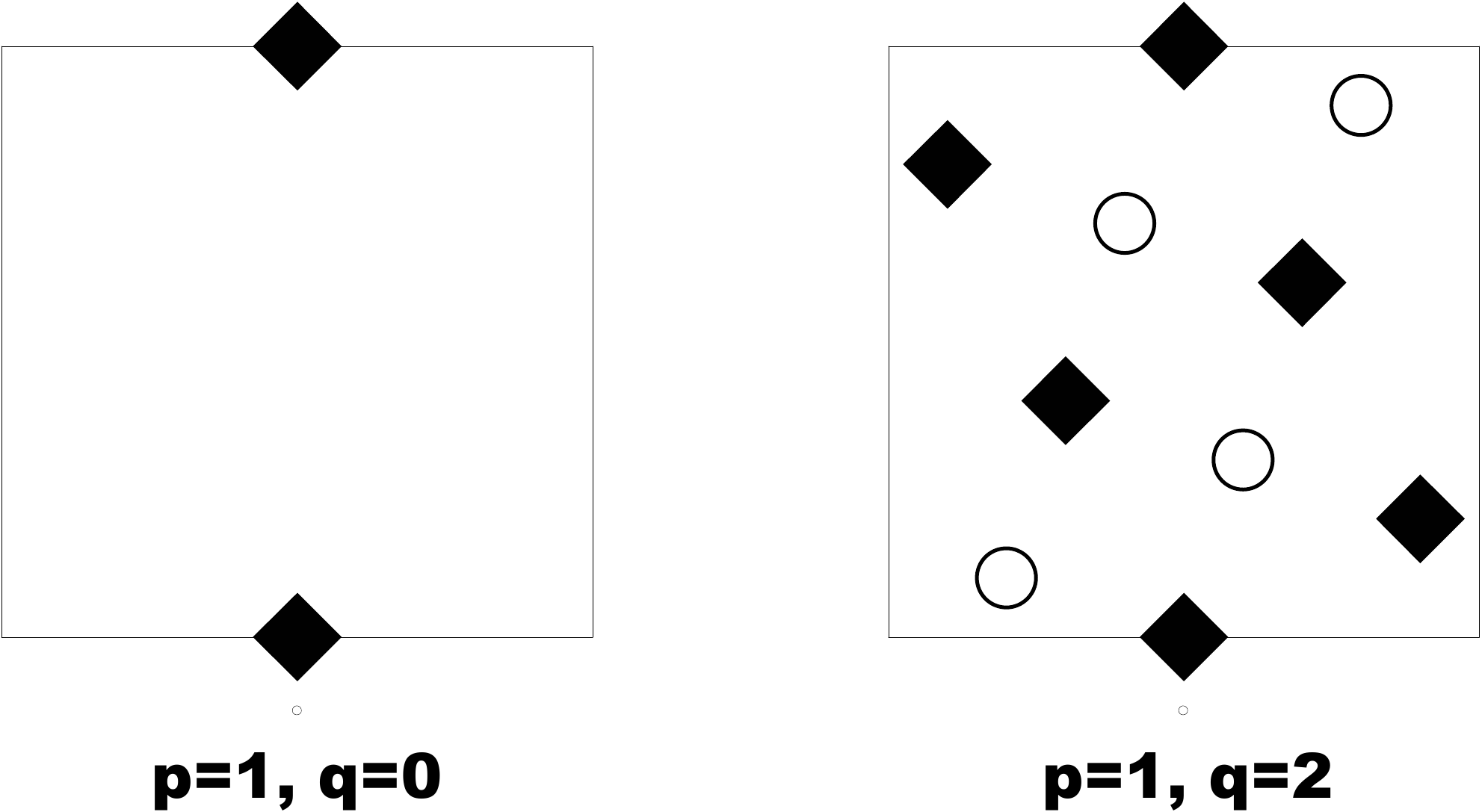}  
  \caption{}
  \label{fig:11}
\end{figure}

In Figure \ref{fig:11} we represent two tiles. 
In the left we consider $p=1$ and $q=0$; 
there are centers located at $\left( \sfrac 12,0\right) $ and 
$\left( \sfrac 12,1\right) $. In the right we consider $p=1$ and $q=2$; there are centers
at $\left( \sfrac 12,0\right) $, $\left( \sfrac 12,1\right) $, 
$\left( \sfrac 3{10},\sfrac 25\right) $, $\left( \sfrac 9{10},\sfrac 15\right) $,
$\left(\sfrac 7{10},\sfrac 35\right) $ and $\left( \sfrac 3{10},\sfrac 25\right) $. We
only represent $p4$ versions of the tiles; there are also $p4mm$ versions
and $p4gm$ versions. We can also consider interchanging $p$ and $q$; the
figures are the same but seen as in a mirror.

\subsection{The fourth type of general tiles}

Take $p$ and $q$ both odd numbers and the square with vertices in $\left(
0,0\right) $, $\frac 12\left( p,q\right) $, $\frac 12\left( p-q,p+q\right) $, 
$\frac 12\left( -q,p\right) $, in the referential of the pattern. A tile
of this type is obtained by intersecting this square with the pattern in the
plane.

The rotation centers of order $4$ are located, in the referential of the
tile, at 
\begin{equation*}
\left( a,b\right) +\frac 2{p^2+q^2}\left( rp+sq,sp-rq\right) \text{,}
\end{equation*}
with $r,s\in \mathbb{Z}$, $\left( a,b\right) =\left( 0,0\right) $ or $\left(
a,b\right) =\left( 1,0\right) $, indifferently. Notice that if there is a
rotation center at $\left( 0,0\right) $, then there is another one at 
$\left( 1,1\right) $; the same happens with $\left( 1,0\right) $ and $\left(
0,1\right) $.

The number of such rotation centers is not an integer; it is 
$\sfrac{(p^2+q^2)}4 $, where $p^2+q^2$ is even.

\begin{figure}[h]
  \centering
  \includegraphics[width=4.0136in]{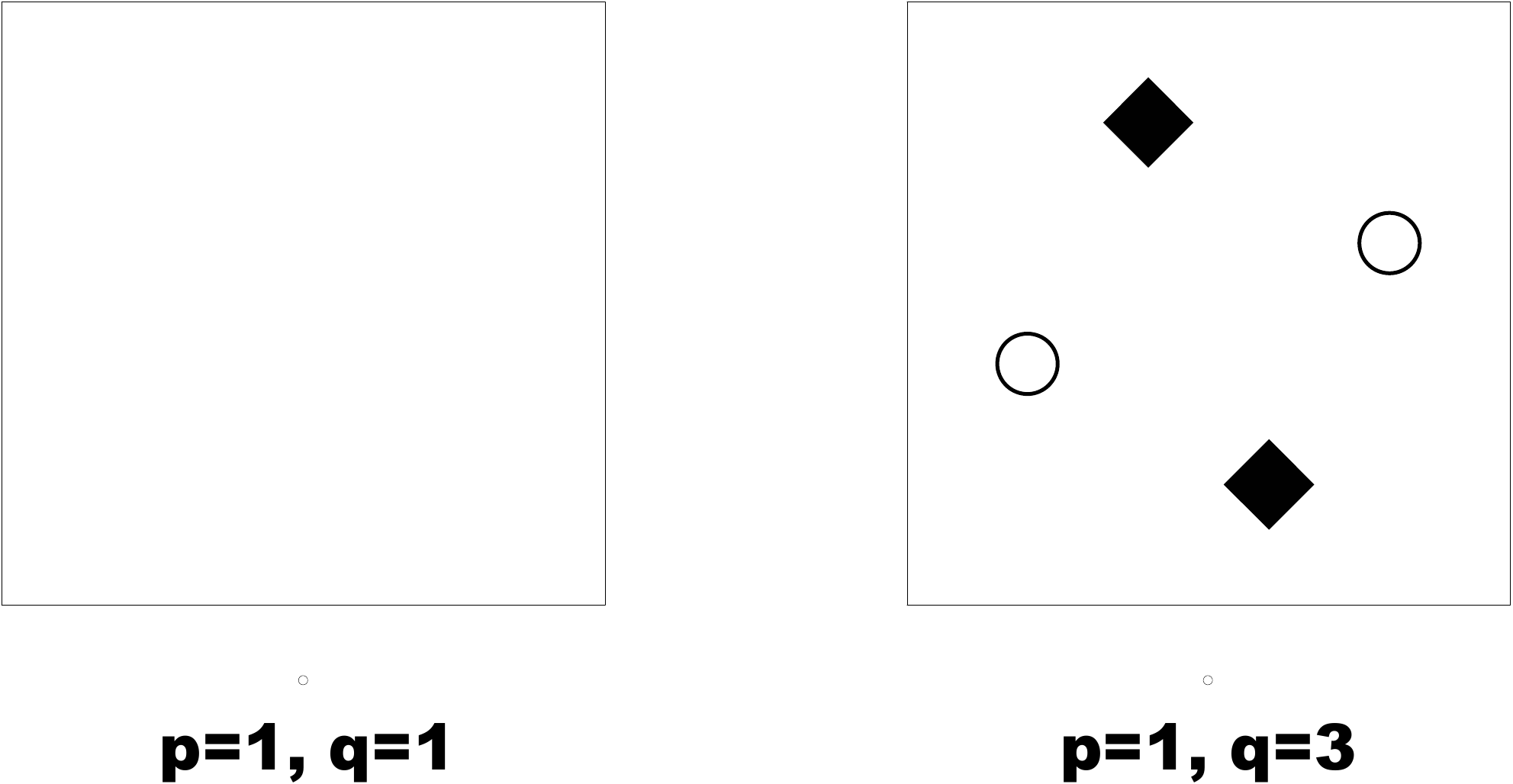}  
  \caption{}
  \label{fig:12}
\end{figure}

In Figure \ref{fig:12} we represent two tiles. 
In the left we consider $p=q=1$; it is
a tile without symmetries. In the right we consider $p=1$ and $q=3$; there
are centers at $\left( \sfrac 35,\sfrac 15\right) $,
$\left( \sfrac 25,\sfrac 45\right) $, 
and the others are located in the vertices
$\left(0,0\right) $ and $\left( 1,1\right) $. We only represent $p4$ versions of
the tiles; there are also $p4mm$ versions and $p4gm$ versions. We can also
consider interchanging $p$ and $q$; the figures are the same but seen as in
a mirror.

\subsection{The exceptions}

There are eight exceptions to these four general rules. All of them have two
rotation centers of order $4$ in the middle of two edges.

In the first type of these tiles these two edges have a common vertex.
Assume that the rotation centers of order $4$ are at $\left( \sfrac
12,0\right) $ and $\left( 0,\sfrac 12\right) $. Then, there are also rotation
centers of order $2$ at $\left( \sfrac 14,\sfrac 14\right) $, $\left( \sfrac
34,\sfrac 14\right) $,$\left( \sfrac 34,\sfrac 34\right) $ and $\left( \sfrac
14,\sfrac 34\right) $. This tile with reflections, has the diagonal that
contains $\left( 0,0\right) $ a reflection line ($p4gm$). This is exactly
the case of Eduardo Nery tile. See Figure \ref{fig:13}.

\begin{figure}[h]
  \centering
  \includegraphics[width=4.0975in]{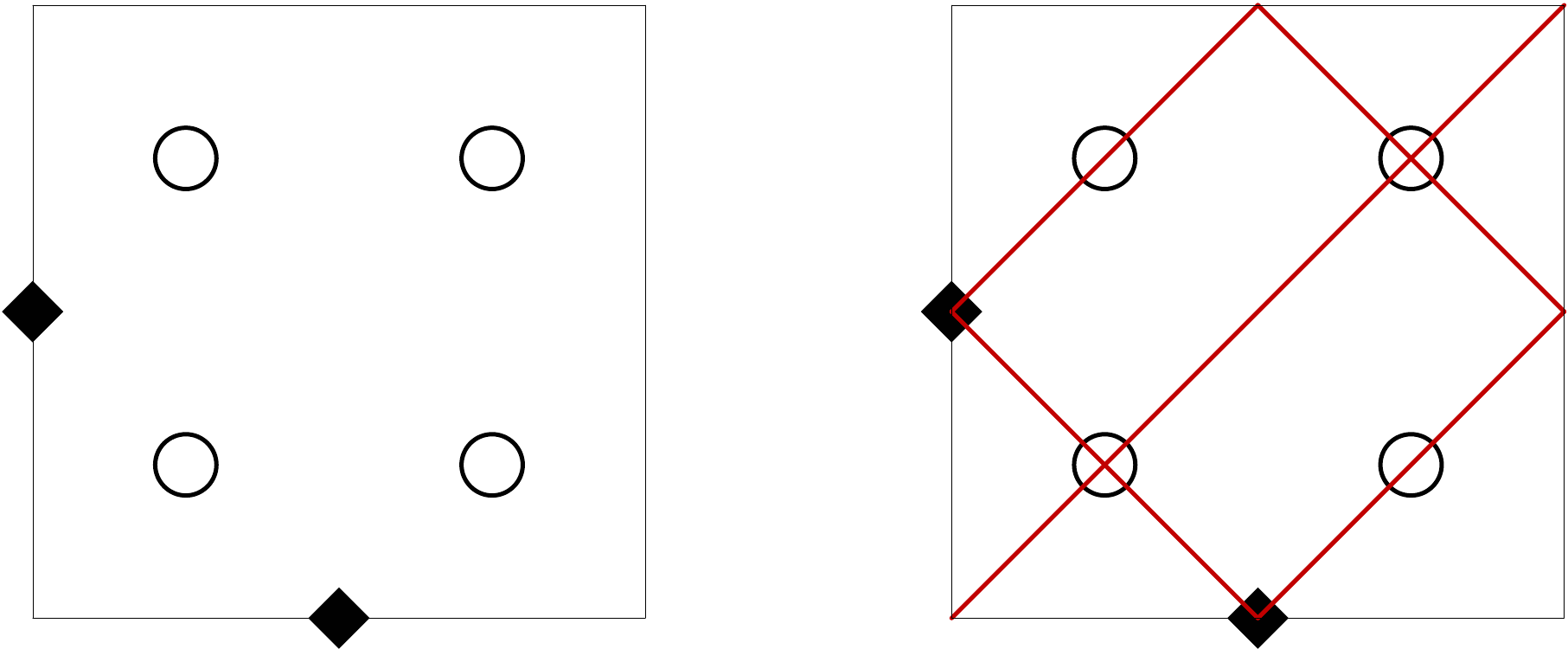}  
  \caption{}
  \label{fig:13}
\end{figure}

In the second and third types of these tiles these two edges have no common
vertex. Assume that the rotation centers of order $4$ are at $\left( \sfrac
12,0\right) $ and $\left( \sfrac 12,1\right) $.

In the second type there is translation invariance by the vectors 
$\left(\sfrac 12,\sfrac 12\right) $ and $\left( \sfrac 12,-\sfrac 12\right) $. This
tile with reflections, has a reflection line containing the points $\left(
\sfrac 12,0\right) $ and $\left( \sfrac 12,1\right) $ ($p4mm$), or a
reflection line containing the points $\left( 0,\sfrac 12\right) $ and 
$\left( 1,\sfrac 12\right) $ ($p4gm$). See Figure \ref{fig:14}.

\begin{figure}[h]
  \centering
  \includegraphics[width=5.5971in]{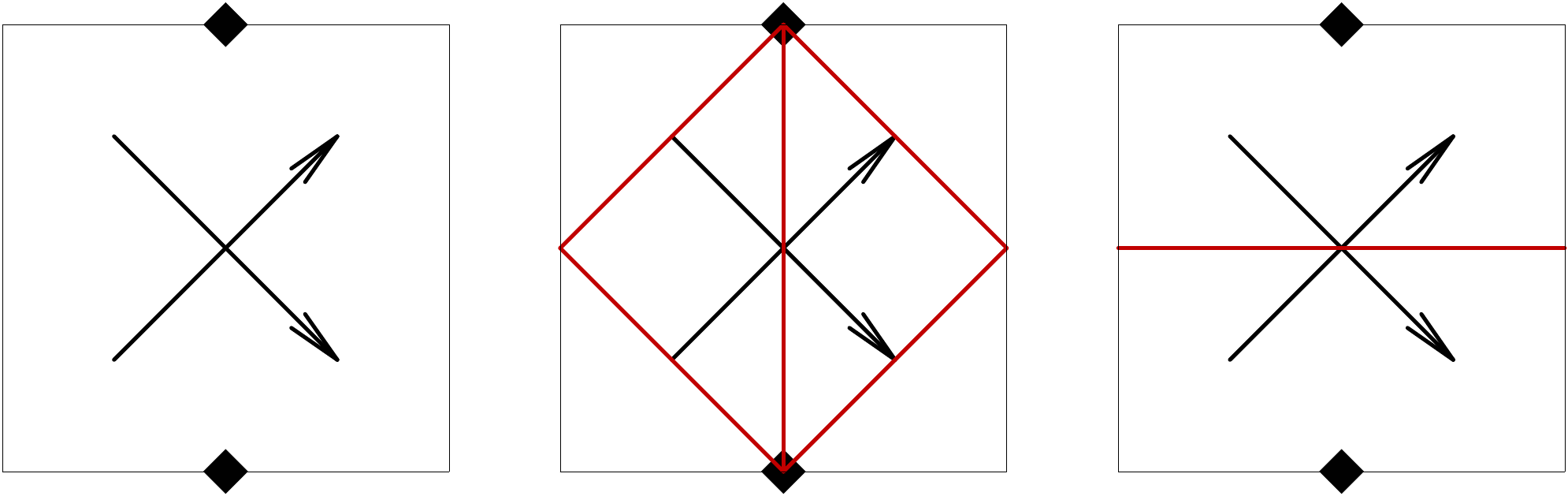}  
  \caption{}
  \label{fig:14}
\end{figure}

In the third type there is no translation invariance and one has rotation
centers of order $2$ at $\left( \sfrac 14,\sfrac 14\right) $, $\left( \sfrac
34,\sfrac 14\right) $, $\left( \sfrac 34,\sfrac 34\right) $ and $\left( \sfrac
14,\sfrac 34\right) $. This tile with reflections, has a reflection line
containing the points $\left( \sfrac 12,0\right) $ and $\left( \sfrac
12,1\right) $ ($p4mm$), or a reflection line containing the points $\left(
0,\sfrac 12\right) $ and $\left( 1,\sfrac 12\right) $ ($p4gm$). 
See Figure \ref{fig:15}.

\begin{figure}[h]
  \centering
  \includegraphics[width=5.5971in]{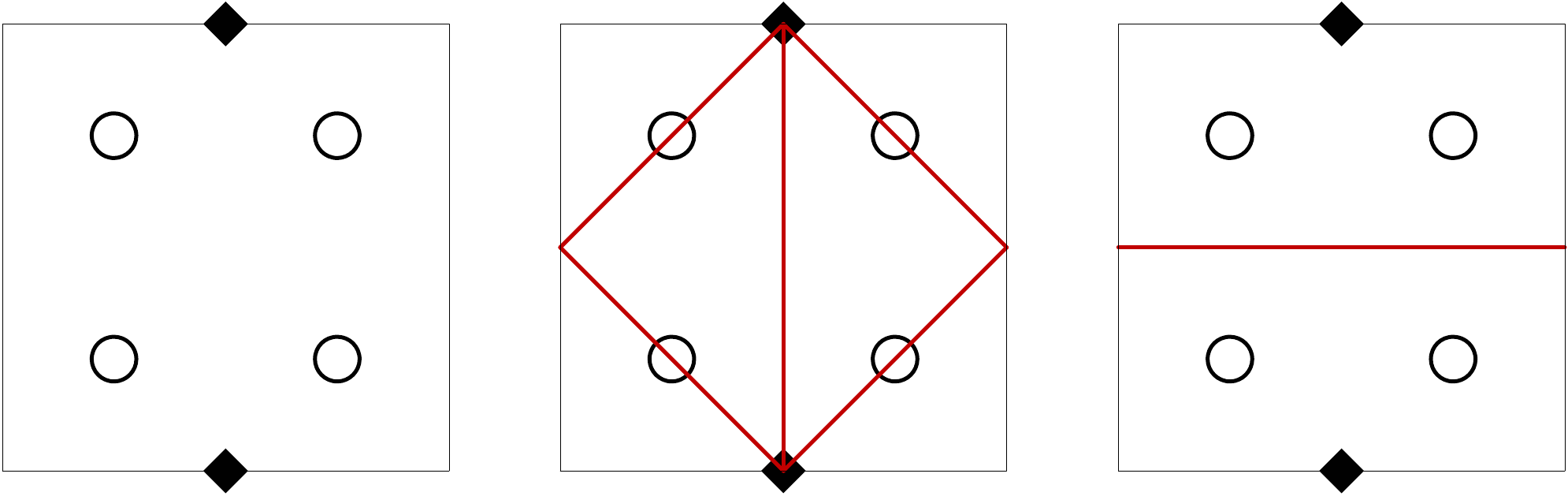}  
  \caption{}
  \label{fig:15}
\end{figure}

\section{Proofs}

\subsection{The general rules}

Consider a tiling of the plane with infinitely many copies of a square tile.
Each tile is located in a region $\left[ p,p+1\right] \times \left[
q,q+1\right] $, with $p,q\in \mathbb{Z}$. This tiling has a non-trivial
rotation center of order $4$ in the point $\left( a,b\right) \in \left[
0,1\right]^2$. Non-trivial means that it is not one of the four vertices of
the tile.

We shall assume w.l.g.\ that $0\leq b\leq a\leq \sfrac 12$, $\,a>0$.

\subsubsection{The action of a rotation on the tiles}

In this section's figures we denote the vertices in the different tiles in the
following way (see Figures \ref{fig:16} and \ref{fig:17}):

\begin{compactenum}[a)]
\item For the tile $\left[ 0,1\right]^2$:\\
$\left( 1,1\right) $ is ``$1$'', $\left( 0,1\right) $ is ``$2$'',
$\left( 0,0\right) $ is ``$3$'' and $\left(1,0\right) $ is ``$4$''.
\item For the tile $\left[ -1,0\right] \times \left[ 0,1\right] $:\\
$\left(0,1\right) $ is ``$a_1$'', $\left( -1,1\right) $ is ``$a_2$'', $\left(
-1,0\right) $ is ``$a_3$'' and $\left( 0,0\right) $ is ``$a_4$''.
\item For the tile $\left[ 0,1\right] \times \left[ -1,0\right]$:\\
$\left(1,0\right) $ is ``$b_1$'', $\left( 0,0\right) $ is ``$b_2$'',
$\left(0,-1\right) $ is ``$b_3$'' and $\left( 1,-1\right) $ is ``$b_4$''.
\item For the tile $\left[ -1,0\right] \times \left[ -1,0\right]$:\\
$\left(0,0\right) $ is ``$c_1$'', $\left( -1,0\right) $ is ``$c_2$'',
$\left(-1,-1\right) $ is ``$c_3$'' and $\left( 0,-1\right) $ is ``$c_4$''.
\item For the tile $\left[ 1,2\right] \times \left[ 0,1\right]$:\\
$\left(2,1\right) $ is ``$d_1$'', $\left( 1,1\right) $ is ``$d_2$'', $\left(
1,0\right) $ is ``$d_3$'' and $\left( 2,0\right) $ is ``$d_4$''.
\item For the tile $\left[ 1,2\right] \times \left[ -1,0\right]$:\\
$\left(2,0\right) $ is ``$e_1$'', $\left( 1,0\right) $ is ``$e_2$'', $\left(
1,-1\right) $ is ``$e_3$'' and $\left( 2,-1\right) $ is ``$e_4$''.
\end{compactenum}

\begin{figure}[h!]
  \centering
  \includegraphics[width=3.8977in]{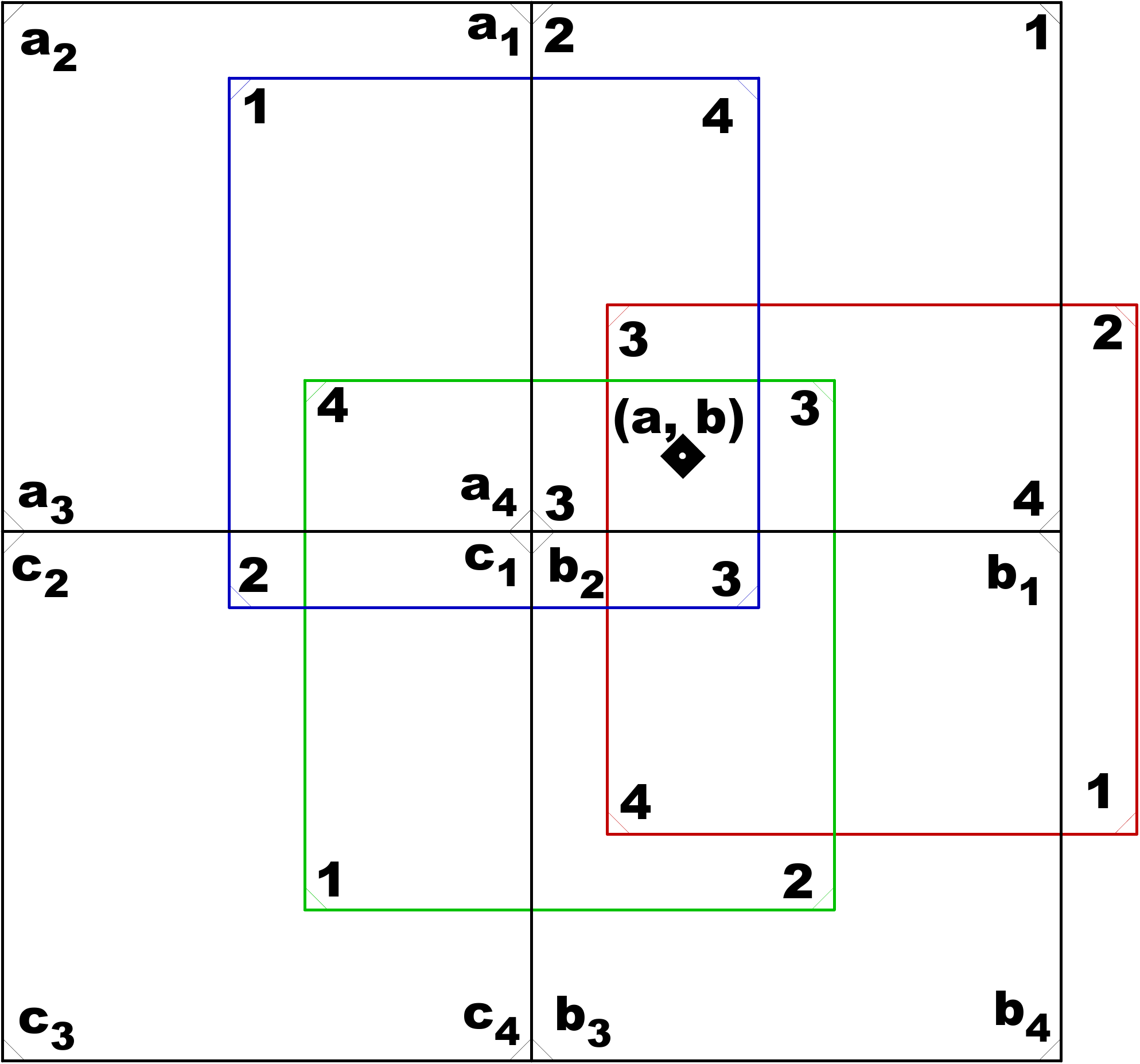}
  \caption{}
  \label{fig:16}
\end{figure}

\begin{figure}[h!]
  \centering
  \includegraphics[width=5.0185in]{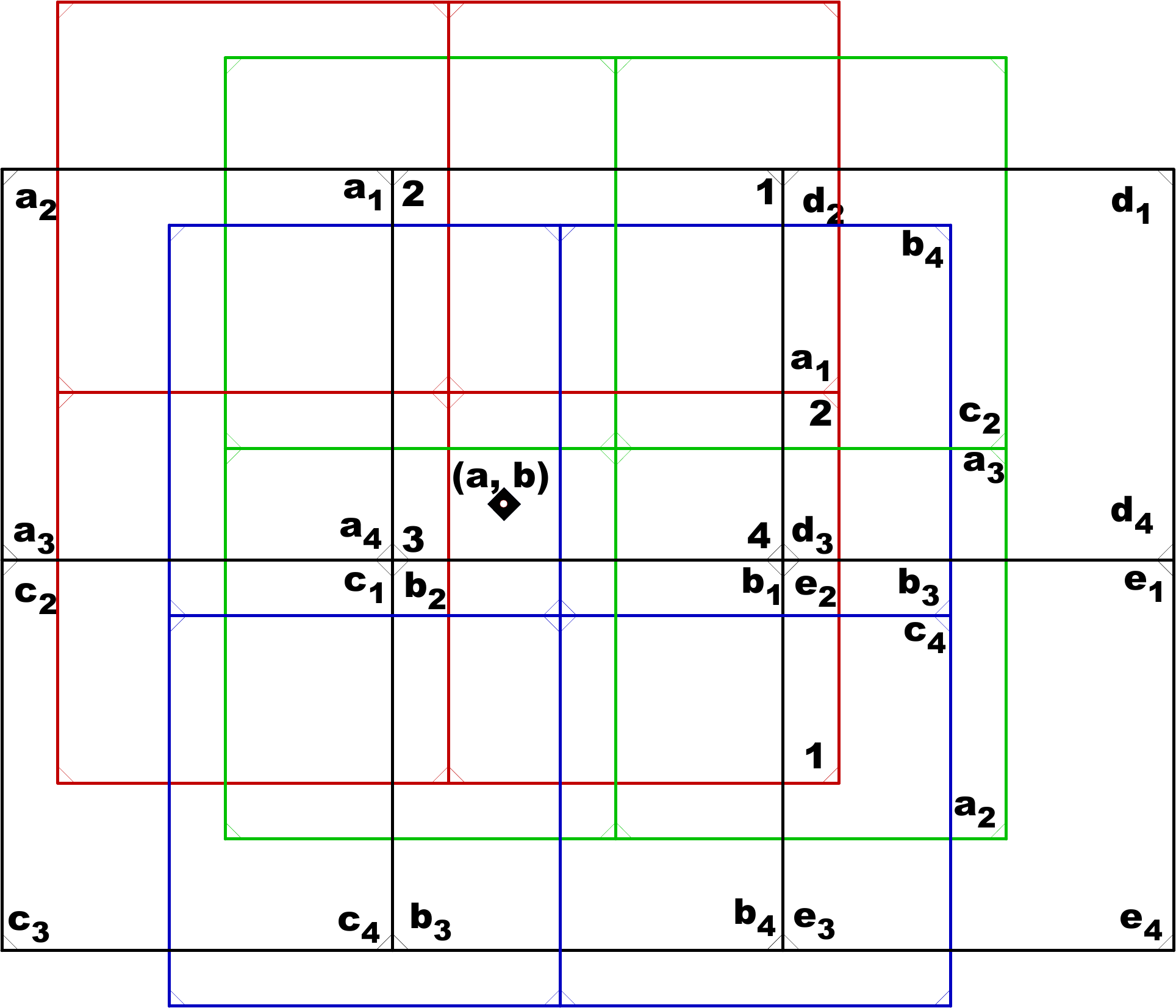}
  \caption{}
  \label{fig:17}
\end{figure}

A single square tile can be located in, at most, four different ways in
every $\left[ p,p+1\right] \times \left[ q,q+1\right] $. We denote A, B, C
and D the different ways of locating the tile in $\left[ -1,0\right] \times
\left[ 0,1\right] $. We denote E, F, G and H the different ways of locating
the tile in $\left[ 0,1\right] \times \left[ -1,0\right]$: \pagebreak[0]
\begin{equation*}
\begin{tabular}{@{\extracolsep{2pt}}@{}ccccc@{}}
\toprule
& A / E & B / F & C / G & D / H \\
\cmidrule(l){2-5}
$a_1$ / $b_1$ & $1$ & $2$ & $3$ & $4$ \\
$a_2$ / $b_2$ & $2$ & $3$ & $4$ & $1$ \\
$a_3$ / $b_3$ & $3$ & $4$ & $1$ & $2$ \\
$a_4$ / $b_4$ & $4$ & $1$ & $2$ & $3$ \\
\bottomrule
\end{tabular}
\end{equation*}

The rotation of order $4$ around $\left( a,b\right) $ generates the
following transformations on the tile (see figures \ref{fig:16}--\ref{fig:24}).
In these tables ``$4$'' means rotation center of order $4$, ``$2$'' means
rotation center of order $2$ and ``tr'' means translation:
\begin{equation*}
\begin{tabular}{@{}ccccc@{}}
\toprule
& A & B & C & D \\
\cmidrule(l){2-5}
$4$ & $\left( a,b\right) $ & $\left( a,b\right) $ & $\left( a,b\right) $ & 
$\left( a,b\right) $ \\
$4$ & $\left( \frac 12+a,\frac 12+b\right) $ & $\left( 1+a-b,a+b\right) $ & 
$\left( \frac 12-b,\frac 12+a\right) $ & $\left( a+b,b-a\right) $ \\
$2$ & $\left( \frac 12+a,b\right) $ & 
$\left( \frac{1+a-b}2,\frac{1+a+b}2\right) $ & - & - \\
tr & - & - & $\left( -2a,1-2b\right) $ & $\left( a-b,a+b\right) $ \\
\bottomrule
\end{tabular}
\end{equation*}
\begin{equation*}
\begin{tabular}{@{}ccccc@{}}
\toprule
& E & F & G & H \\
\cmidrule(l){2-5}
$4$ & $\left( a,b\right) $ & $\left( a,b\right) $ & $\left( a,b\right) $ & 
$\left( a,b\right) $ \\ 
$4$ & $\left( \frac 12+a,\frac 12+b\right) $ & $\left( a-b,a+b\right) $ & 
$\left( \frac 12+b,\frac 12-a\right) $ & $\left( a+b,1-a+b\right) $ \\ 
$4$ & $\left( a-\frac 12,\frac 12+b\right) $ & - & $\left( \frac
12-b,a-\frac 12\right) $ & - \\ 
$2$ & $\left( a,\frac 12+b\right) $ & $\left( \frac{a-b}2,\frac{a+b}2\right) 
$ & - & $\left( \frac{1+a+b}2,\frac{1-a+b}2\right) $ \\ 
tr & - & $\left( a+b,b-a\right) $ & $\left( 2a-1,2b\right) $ & $\left(
1-a+b,1-a-b\right) $ \\
\bottomrule
\end{tabular}
\end{equation*}

Each set of transformations, in these cases (A, B, C, D, E, F, G and H),
generates a $p_4$ pattern in the plane with $pu_i+qv_i$ translation
invariance, where $u_i$ and $v_i$ are
\begin{align*}
&\begin{tabular}{@{}ccccc@{}}
$i\rightarrow $ & A & B & C & D \\
\midrule
$u_i$ & $\left( \frac 12,\frac 12\right) $ & $\left( 1-b,a\right) $ & 
$\left( \frac 12-a-b,\frac 12+a-b\right) $ & $\left( a,b\right) $ \\ 
$v_i$ & $\left( -\frac 12,\frac 12\right) $ & $\left( -a,1-b\right) $ & 
$\left( -\frac 12-a+b,\frac 12-a-b\right) $ & $\left( -b,a\right) $ \\ 
\end{tabular}\\[10pt]
&\begin{tabular}{@{}ccccc@{}}
$i\rightarrow $ & E & F & G & H \\ 
\midrule
$u_i$ & $\left( \frac 12,\frac 12\right) $ & $\left( a,b\right) $ & $\left(
\frac 12-a+b,\frac 12-a-b\right) $ & $\left( b,1-a\right) $ \\ 
$v_i$ & $\left( -\frac 12,\frac 12\right) $ & $\left( -b,a\right) $ & 
$\left( a+b-\frac 12,\frac 12-a+b\right) $ & $\left( a-1,b\right) $ \\ 
\end{tabular}
\end{align*}

For the tile located at $\left[ -1,0\right] \times \left[ -1,0\right] $, the
results are 
\begin{equation*}
\begin{tabular}{@{}ccccc@{}}
\toprule
& $c_1=1$ & $c_1=2$ & $c_1=3$ & $c_1=4$ \\ 
\cmidrule(l){2-5}
$4$ & $\left( a,b\right) $ & $\left( a,b\right) $ & $\left( a,b\right) $ & 
$\left( a,b\right) $ \\ 
$4$ & $\left( a,1+b\right) $ & $\left( \frac 12+a-b,\frac 12+a+b\right) $ & 
$\left( -b,a\right) $ & $\left( \frac 12+a+b,\frac 12-a+b\right) $ \\ 
$2$ & $\left( \frac 12+a,\frac 12+b\right) $ & 
$\left( \frac{a-b}2,\frac{1+a+b}2\right) $ & - & - \\ 
tr & - & - & $\left( 2a,2b\right) $ & $\left( a-b-1,a+b\right) $ \\ 
\bottomrule
\end{tabular}
\end{equation*}
\begin{equation*}
\begin{tabular}{@{}ccccc@{}}
$i\rightarrow $ & $\left( c_1=\right) 1$ & $\left( c_1=\right) 2$ & $\left(
c_1=\right) 3$ & $\left( c_1=\right) 4$ \\
\midrule
$u_i$ & $\left( 1,0\right) $ & $\left( \frac 12-b,\frac 12+a\right) $ & 
$\left( a-b,a+b\right) $ & $\left( \frac 12+b,\frac 12-a\right) $ \\ 
$v_i$ & $\left( 0,1\right) $ & $\left( -\frac 12-a,\frac 12-b\right) $ & 
$\left( -a-b,a-b\right) $ & $\left( a-\frac 12,\frac 12+b\right) $ \\ 
\end{tabular}
\end{equation*}


\begin{figure}[htb!]
  \centering
  \includegraphics[width=4.8516in]{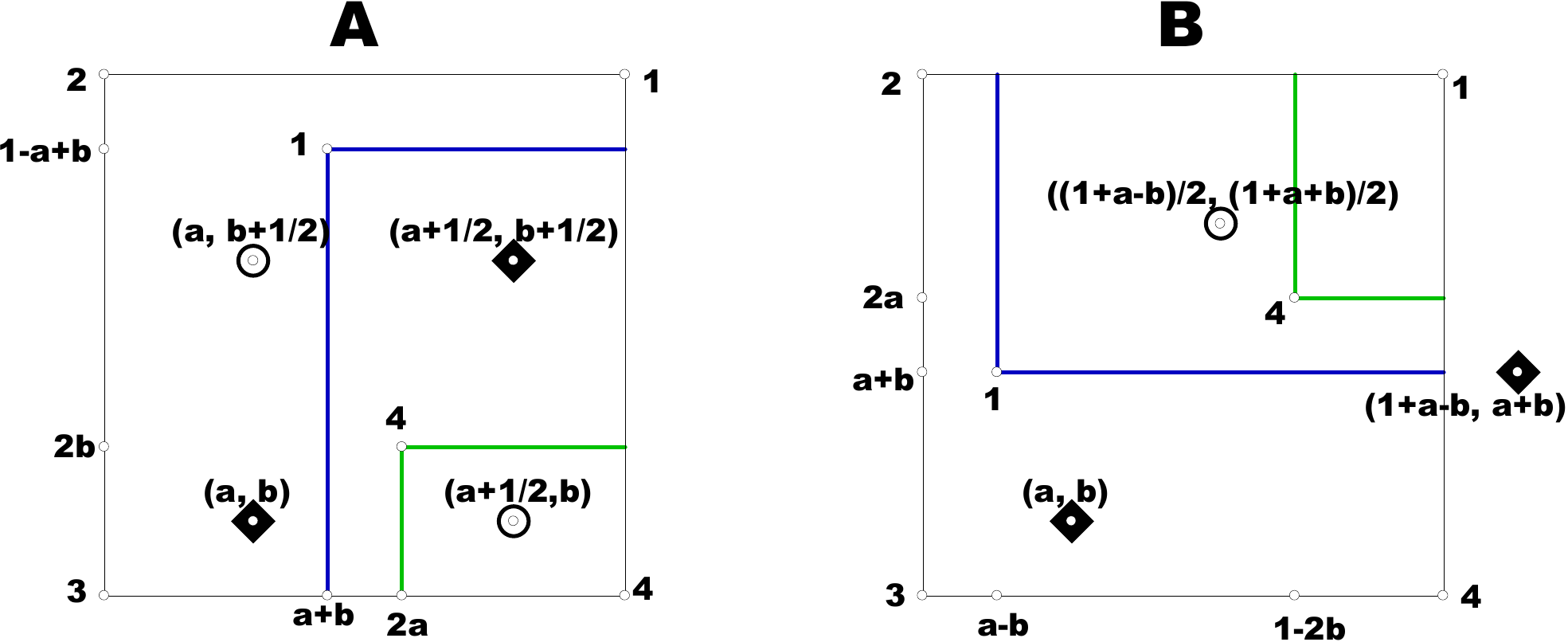}
  \vspace*{-8pt}
  \caption{}
  \label{fig:18}
\end{figure}


\begin{figure}[htb!]
  \centering
  \includegraphics[width=4.5619in]{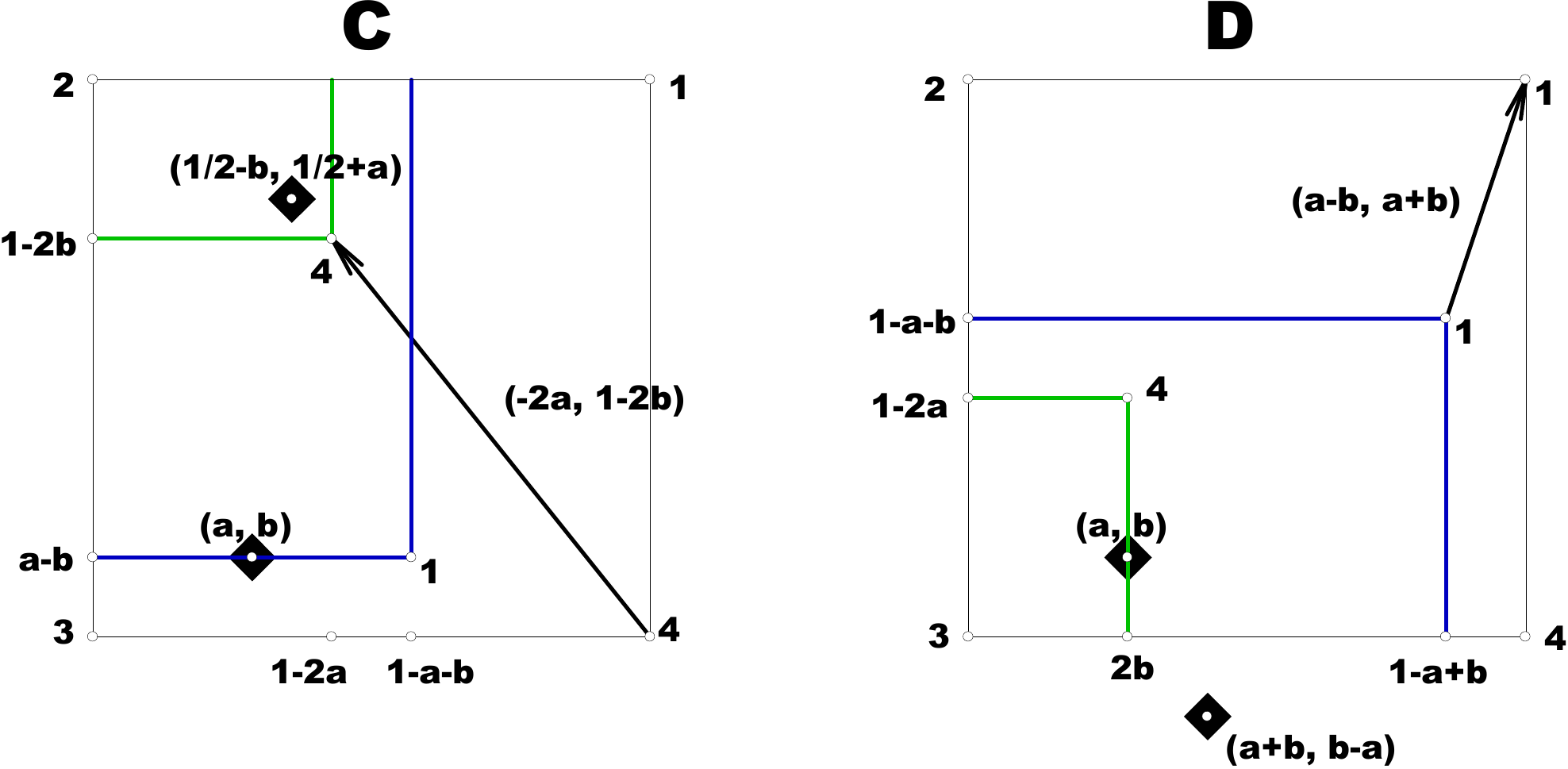}
  \vspace*{-16pt}
  \caption{}
  \label{fig:19}
\end{figure}

Notice that there are more transformations involving, for example, the
interaction of the tiles $\left[ -1,0\right] \times \left[ 0,1\right] $ and 
$\left[ 0,1\right] \times \left[ -1,0\right] $:

\begin{equation*}
\begin{tabular}{@{}ccccc@{}}
\toprule
& AE & AF & AG & AH \\ 
\cmidrule(l){2-5}
$4$ & $\left( a,b\right) $ & $\left( \frac 12+a-b,\frac 12+a+b\right) $ & 
$\left( 1-b,a\right) $ & $\left( \frac 12+a+b,\frac 12-a+b\right) $ \\ 
$4$ & $\left( a,1+b\right) $ & $\left( -b,a\right) $ & $\left( 1-a,-b\right) 
$ & $\left( 1+b,1-a\right) $ \\ 
$2$ & $\left( \frac 12+a,\frac 12+b\right) $ & 
$\left( \frac{1+a-b}2,\frac{a+b}2\right) $ & - & - \\ 
tr & - & - & $\left( 2a,2b\right) $ & $\left( -a+b,1-a-b\right) $ \\ 
\bottomrule
\end{tabular}
\end{equation*}

\begin{equation*}
\begin{tabular}{@{\extracolsep{-1.2pt}}@{}ccccc@{}}
\toprule
& BE & BF & BG & BH \\ 
\cmidrule(l){2-5}
$4$ & $\left( \frac 12+a-b,\frac 12+a+b\right) $ & $\left( 1-b,a\right) $ & 
$\left( \frac 12-a-b,\frac 12+a-b\right) $ & $\left( 1-a,1-b\right) $ \\ 
$4$ & $\left( a,1+b\right) $ & $\left( -b,a\right) $ & $\left( 1-a,-b\right) 
$ & $\left( 1+b,1-a\right) $ \\ 
$2$ & - & $\left( \frac 12-b,\frac 12+a\right) $ & 
$\left( 1-\frac{a+b}2,\frac{1+a-b}2\right) $ & - \\ 
tr & $\left( 1-a-b,a-b\right) $ & - & - & $\left( -2b,2a\right) $ \\ 
\bottomrule
\end{tabular}
\end{equation*}

\begin{equation*}
\begin{tabular}{@{}ccccc@{}}
\toprule
& CE & CF & CG & CH \\ 
\cmidrule(l){2-5}
$4$ & $\left( b,1-a\right) $ & $\left( \frac 12-a-b,\frac 12+a-b\right) $ & 
$\left( 1-a,1-b\right) $ & $\left( \frac 12-a+b,\frac 12-a-b\right) $ \\ 

$4$ & $\left( a,1+b\right) $ & $\left( -b,a\right) $ & $\left( 1-a,-b\right) 
$ & $\left( 1+b,1-a\right) $ \\ 
$2$ & - & - & $\left( \frac 12-a,\frac 12-b\right) $ & 
$\left( \frac{1-a+b}2,1-\frac{a+b}2\right) $ \\ 
tr & $\left( 2a,2b\right) $ & $\left( -a+b,1-a-b\right) $ & - & - \\ 
\bottomrule
\end{tabular}
\end{equation*}

\begin{equation*}
\begin{tabular}{@{}ccccc@{}}
\toprule
& DE & DF & DG & DH \\
\cmidrule(l){2-5}
$4$ & $\left( \frac 12+a+b,\frac 12-a+b\right) $ & $\left( a,b\right) $ & 
$\left( \frac 12-a+b,\frac 12-a-b\right) $ & $\left( b,1-a\right) $ \\ 
$4$ & $\left( a,1+b\right) $ & $\left( -b,a\right) $ & $\left( 1-a,-b\right) 
$ & $\left( 1+b,1-a\right) $ \\ 
$2$ & $\left( \frac{a+b}2,\frac{1-a+b}2\right) $ & - & - & $\left( \frac
12+b,\frac 12-a\right) $ \\ 
tr & - & $\left( -2b,2a\right) $ & $\left( 1-a-b,a-b\right) $ & - \\ 
\bottomrule
\end{tabular}
\end{equation*}

\begin{equation*}
\begin{tabular}{@{}ccccc@{}}
$i\rightarrow $ & AE/BF/CG/DH & AF/BG/CH/DE & AG/BH/CE/DF & AH/BE/CF/DG \\ 
\midrule
$u_i$ & $\left( 1,0\right) $ & $\left( \frac 12+a,\frac 12+b\right) $ & 
$\left( a-b,a+b\right) $ & $\left( \frac 12-a,\frac 12-b\right) $ \\ 
$v_i$ & $\left( 0,1\right) $ & $\left( -\frac 12-b,\frac 12+a\right) $ & 
$\left( -a-b,a-b\right) $ & $\left( b-\frac 12,\frac 12-a\right) $ \\ 
\end{tabular}
\end{equation*}

\subsubsection{Case E}

In this case the transformations are the three rotations of order $4$ 
($\sigma_0$, $\sigma_1$ and $\sigma_2$) with centers located at 
$\left(a,b\right) $, $\left( \sfrac 12+a,\sfrac 12+b\right) $ and
$\left( a-\sfrac 12,\sfrac 12+b\right) $, and the rotation of order $2$ ($\sigma_3$) with
centers located at $\left( a,\sfrac 12+b\right) $. These transformations are
compatible in the sense that they generate a pattern $p4$ in the plane.

\begin{figure}[htb!]
  \centering
  \includegraphics[width=4.907in]{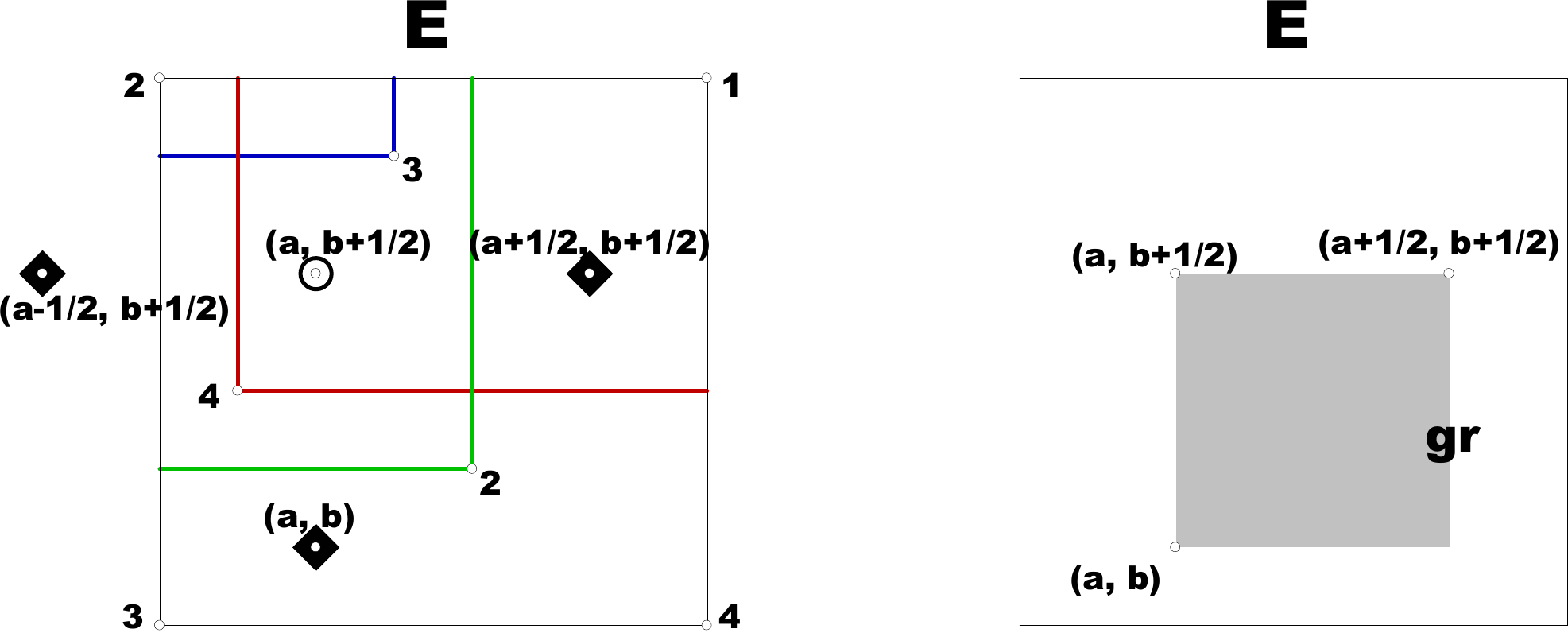}  
  \caption{}
  \label{fig:20}
\end{figure}

\begin{proposition}
The tile is such that it has the pattern $p4$ generated by the four
transformations in the plane ($\mathbb{R}^2$).
\end{proposition}

\begin{proof}
Consider the dashed square \textbf{gr} contained in the tile 
(Figure \ref{fig:20}). 
\textbf{gr} is a generating region of the pattern $p4$ in the plane. Its
rotation by $\sigma_3$ covers $\left[ 0,a\right] \times \left[ b+\sfrac
12,1\right] $. Its rotation by $\sigma_0$, $\sigma_0^2$ and $\sigma_0^3$
covers $\left[ 0,a\right] \times \left[ 0,b+\sfrac 12\right] $ and $\left[
a,a+\sfrac 12\right] \times \left[ 0,b\right] $. Its rotation by $\sigma_1$, 
$\sigma_1^2$ and $\sigma_1^3$ covers $\left[ a+\sfrac 12,1\right] \times
\left[ b,1\right] $ and $\left[ a,a+\sfrac 12\right] \times \left[ b+\sfrac
12,1\right] $. The rotation of $\left[ a,a+b\right] \times \left[ b+\sfrac
12,1-a+b\right] $ by $\sigma_1$ covers $\left[ a+\sfrac 12,1\right] \times
\left[ 0,b\right] $.
\end{proof}

\subsubsection{Case F}

In this case the transformations are a translation by 
$u=(a+b,\linebreak[0]
b-a)$, 
two rotations of order $4$ ($\sigma_0$ and $\sigma_1$)
with centers located at $\left( a,b\right) $ and $\left( a-b,a+b\right) $,
and a rotation of order $2$ ($\sigma_2$) with center located at $\bigl( 
\frac{a-b}2,\frac{a+b}2\bigr) $. These transformations are compatible in
the sense that they generate a pattern $p4$ in the plane.

\begin{figure}[htb!]
  \centering
  \includegraphics[width=4.6561in]{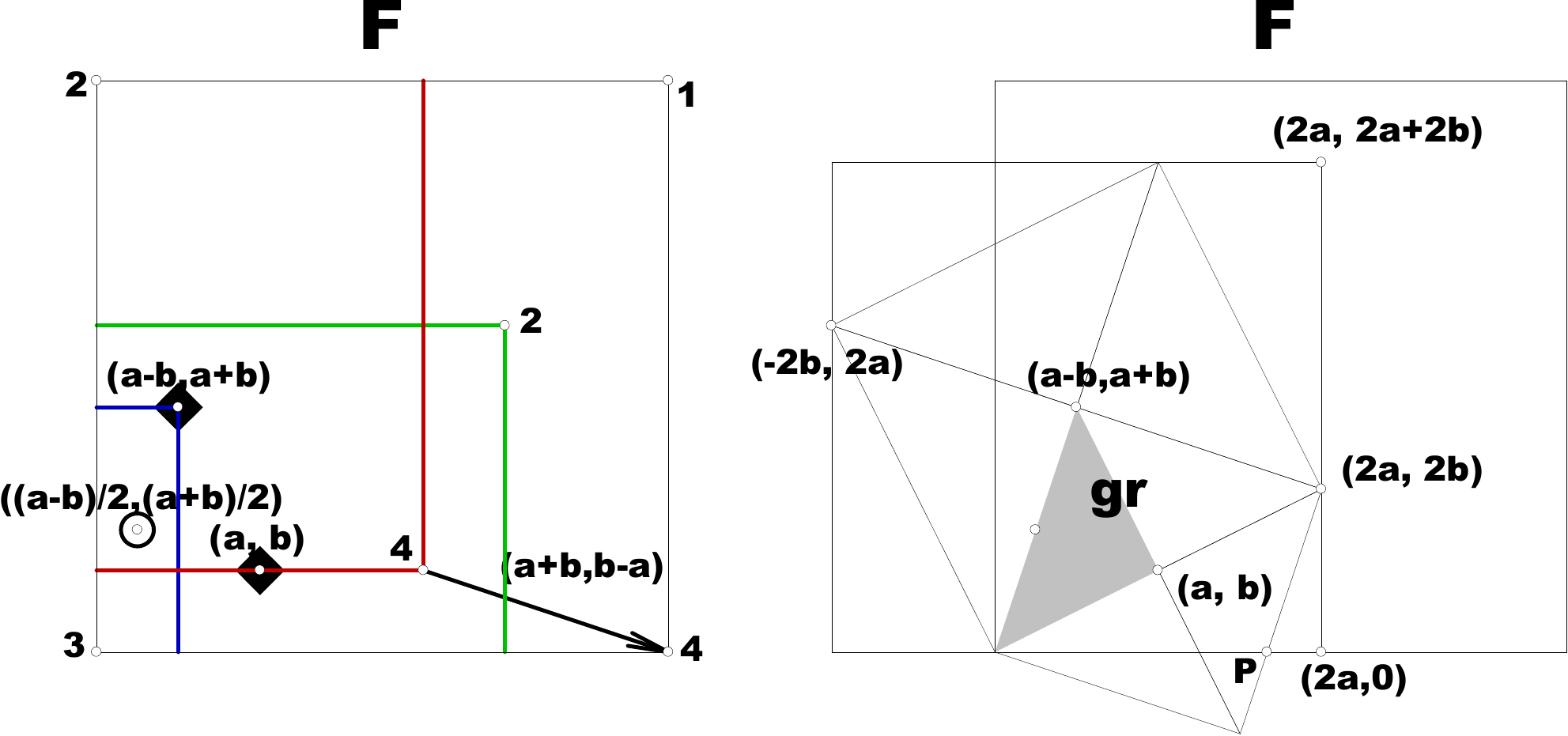}
  \vspace*{-8pt}
  \caption{}
  \label{fig:21}
\end{figure}

\begin{proposition}
The covered region of the tile, by the pattern $p4$ in the plane, is the
union of $\left[ 2a,1\right] \times \left[ 0,a+b\right] $ and the
translations by $nu$ of $\left[ 0,2a\right] \times \left[ 0,1\right] $ and
(whenever $a+b<\frac 12$ and $b<a$) by $-nu$ of $\left[ 2a,1-2b\right]
\times \left[ a+b,2a\right] $, $n=0,1,2,\ldots $
\end{proposition}

\begin{proof}
Consider the dashed triangle \textbf{gr} contained in the tile 
(Figure \ref{fig:21}); 
\textbf{gr} is a generating region of the pattern $p4$ in the plane. Its
rotations by $\sigma_0$ covers the quadrilateral with vertices at 
$\left(0,0\right) $, 
$P$, 
$\left( 2a,2b\right) $ and 
$\left( a-b,a+b\right) $; 
$P=\left(\sfrac{2(a^2+b^2)}{(a+b)},0\right) $. 
The triangle with vertices at $P$, 
$\left( 2a,0\right) $ and 
$\left( 2a,2b\right) $ is the
translation by $u$ of the triangle with vertices 
$\left( \sfrac{(a-b)^2}{(a+b)},a-b\right) $, 
$\left( a-b,a-b\right) $ and 
$\left(a-b,a+b\right) $. 
Hence, we have covered a quadrilateral with vertices at 
$\left( 0,0\right) $, 
$\left( 2a,0\right) $, 
$\left( 2a,2b\right) $ and 
$\left( a-b,a+b\right) $. Its rotation by $\sigma_1$, 
$\sigma_1^2$ and 
$\sigma_1^3$ covers the rectangle 
$\left[ 0,2a\right] \times \left[0,2a+2b\right] $.

If one rotates this last rectangle by $\sigma_0^{-1}$, one covers $\left[
2a,3a+b\right] \times \left[ 0,a+b\right] $. Let us use the notations 
\begin{equation*}
\begin{split}
\alpha_n &=\left( n+3\right) a+\left( n+1\right) b,\quad\text{for }0\leq n\leq 
\frac{1-3a-b}{a+b}\text{,} \\
&=1,\quad\text{for }n\geq \frac{1-3a-b}{a+b}\text{;}
\end{split}
\end{equation*}
\begin{equation*}
\begin{split}
\gamma_n &=\left( n+2\right) a+\left( n+2\right) b,\quad\text{for }0\leq n\leq 
\frac{1-2a-2b}{a+b}\text{,} \\
&=1,\quad\text{for }n\geq \frac{1-2a-2b}{a+b}\text{,}
\end{split}
\end{equation*}
and call $\mathcal{R}_n$ the region union of $\left[ 0,2a\right] \times
\left[ 0,\gamma_n\right] $ and $\left[ 2a,\alpha_n\right] \times \left[
0,a+b\right] $. This region of the tile has the pattern $p4$ of the plane
for $n=0$.

Assume that, for some $n\geq 0$, the region $\mathcal{R}_n$ is covered.

Rotating $\mathcal{R}_n$ by $\sigma_1$, one covers also the rectangle 
$[ a-b,2a] \times [ \gamma_n,\gamma_{n+1}] $. The
rectangle $[ 0,a-b] \times [ \gamma_n,\gamma_{n+1}] $
is the translation by $( -u) $ of 
$[a+b,2a] \times [\gamma_n-a+b,\gamma_{n+1}-a+b]$. 
Hence, the region $[0,2a] \times [ \gamma_n,\gamma_{n+1}] $ is covered,
and so is the
region $[ 0,2a] \times [ 0,\gamma_{n+1}] $.

If one rotates this last rectangle by $\sigma_0^{-1}$, one covers $[
2a,\alpha_{n+1}] \times [ 0,a+b] $. Hence, 
$\mathcal{R}_{n+1}$ is covered.

Notice that if $\alpha_n=1$, then $\gamma_{n+1}=1$, and if $\gamma
_{n+1}=1 $, then $\alpha_n=1$. Therefore, for some 
$n\geq 0$, $\mathcal{R}_n $ is the union of 
$[ 0,2a] \times [ 0,1] $ and 
$[ 2a,1] \times [ 0,a+b] $.

If $a+b<\sfrac 12$ and $b<a$, then rotating $[ 0,2a] \times [
0,1] $ by $\sigma_1^{-1}$, one covers $[ 2a,1-2b] \times
[ a+b,2a] $.

Finally, the covered region of the tile, by the pattern $p4$ in the plane,
is the union of $[ 2a,1] \times [ 0,a+b] $ and the
translations by $nu$ of $[ 0,2a] \times [ 0,1] $ and
(whenever $a+b<\sfrac 12$ and $b<a$) by $-nu$ of $[ 2a,1-2b]
\times [ a+b,2a] $, $n=0,1,2,\ldots $
\end{proof}

Notice that the covered region contains 
\begin{equation*}
\left\{ \left( x,y\right) \in \left[ 0,1\right]^2:\,y\leq 1+\left( a+b\right)
^{-1}\left( b-a\right) \left( x-a+b\right) \right\} \text{.}
\end{equation*}

\subsubsection{Case G}

Assume that $\left( a,b\right) \neq \left( \sfrac 12,0\right) $. In this case
the transformations are a translation by $u=\left( 1-2a,-2b\right) $, and
three rotations of order $4$ ($\sigma_0$, $\sigma_1$ and $\sigma_2$) with
centers located at $\left( a,b\right) $, $\left( \sfrac 12+b,\sfrac
12-a\right) $ and $\left( \sfrac 12-b,a-\sfrac 12\right) $ 
(see Figure \ref{fig:22}).
These transformations are compatible in the sense that they generate a
pattern $p4$ in the plane.

\begin{figure}[htb!]
  \centering
  \includegraphics[width=4.7262in]{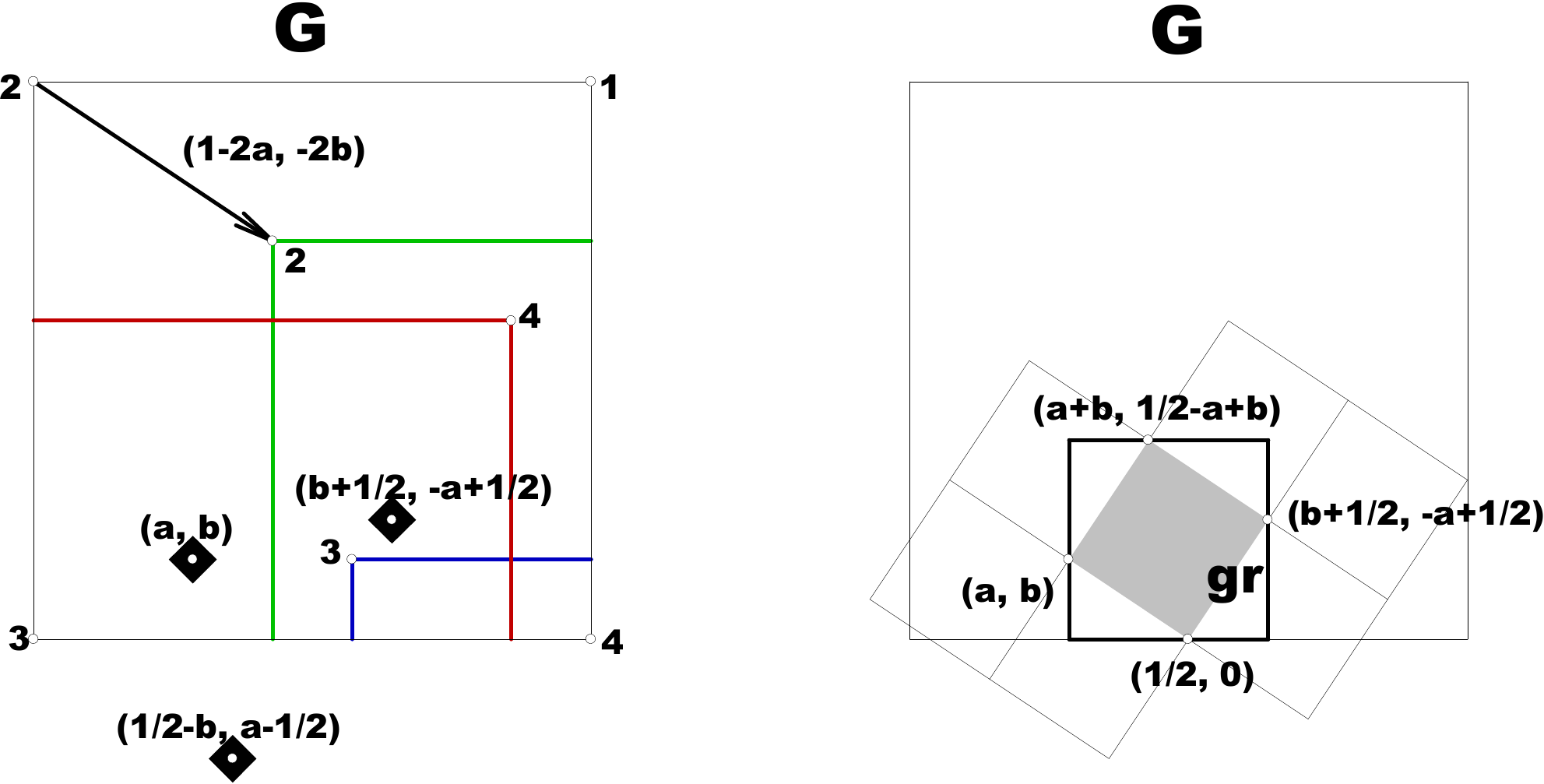}
  \vspace*{-8pt}
  \caption{}
  \label{fig:22}
\end{figure}

Notice that if $\left( a,b\right) =\left( \sfrac 12,0\right) $, the
transformations collapse: $u=0$ and all the centers are located at $\left(
\sfrac 12,0\right) $.

\begin{figure}[htb!]
  \centering
  \vspace{8pt}
  \includegraphics[width=5.5556in]{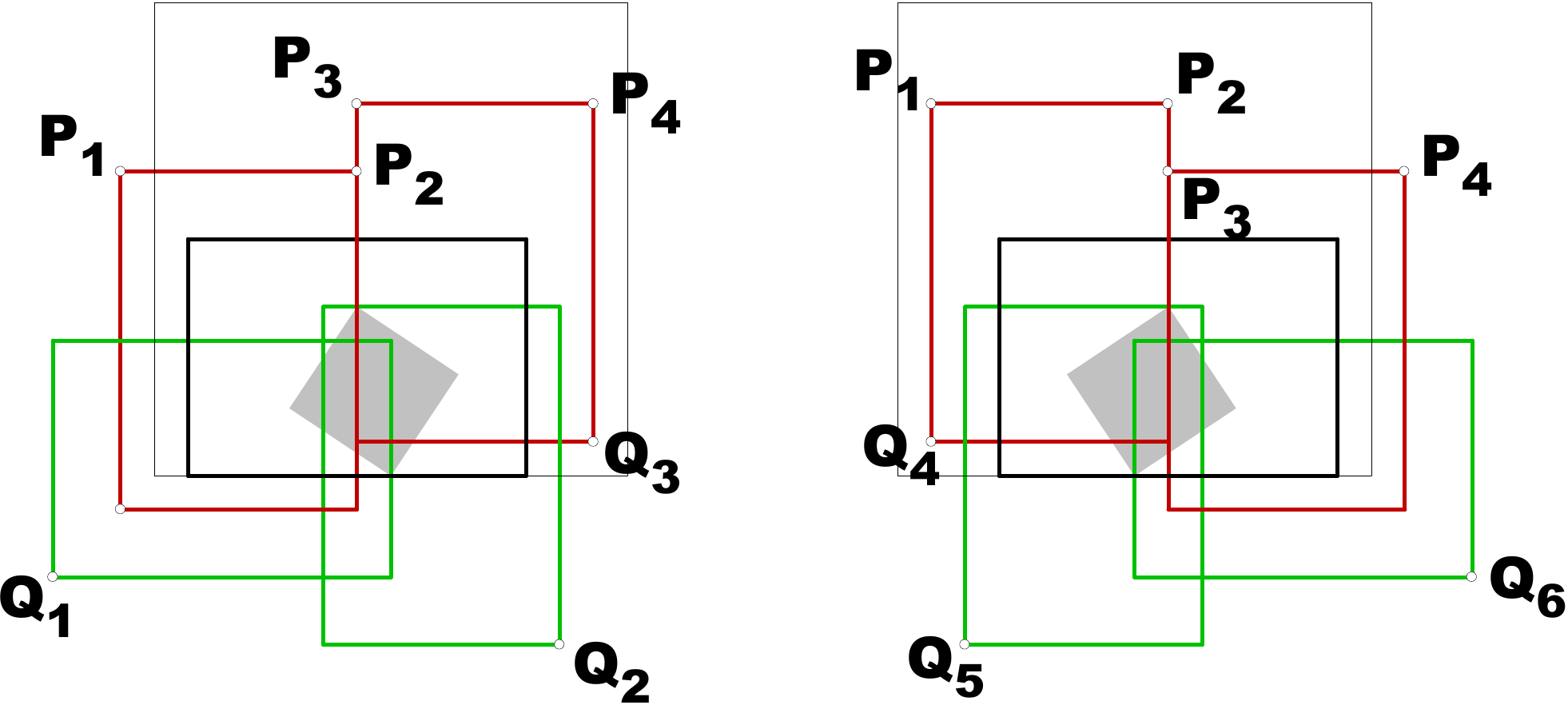}
  \vspace*{-4pt}
  \caption{}
  \label{fig:23}
\end{figure}

In Figure \ref{fig:23} 
the black rectangle, $\mathcal{R}_n$, has vertices at $\left(
\alpha_n,0\right) $, $\left( \beta_n,0\right) $, $\left( \alpha_n,\gamma
_n\right) $ and $\left( \beta_n,\gamma_n\right) $, with 
\begin{align*}
0\leq \alpha_n\leq a<\frac 12+b\leq \beta_n\leq 1\text{,}&\\[2pt]
\text{and \;}0<\frac 12-a+b\leq \gamma_n\leq 1\text{.}&
\end{align*}

The red and the green rectangles are rotations of the black one by $\sigma_0
$, $\sigma_0^2$, $\sigma_1^{-1}$ and $\sigma_1$, in the left hand side
(of Figure \ref{fig:23}), 
and by $\sigma_0$, $\sigma_0^{-1}$, $\sigma_1^{-1}$ and 
$\sigma_1^2$, in the right hand side (of Figure \ref{fig:23}). 
The points $P_1,\dots,P_4$, $Q_1,\dots,Q_6$ have coordinates: 
\begin{align*}
P_1&=\left( a+b-\gamma_n,-a+b+\beta_n\right) \\
P_2&=\left( a+b,-a+b+\beta_n\right) \\
P_3&=\left( a+b,1-a+b-\alpha_n\right) \\
P_4&=\left( a+b+\gamma_n,1-a+b-\alpha_n\right) \\
Q_1&=\left( 2a-\beta_n,2b-\gamma_n\right) \\
Q_2&=\left( 1-a+b,-a-b+\alpha_n\right) \\
Q_3&=\left( a+b+\gamma_n,1-a+b-\beta_n\right) \\
Q_4&=\left( a+b-\gamma_n,-a+b+\alpha_n\right) \\
Q_5&=\left( a-b,a+b-\beta_n\right) \\
Q_6&=\left( 1+2b-\alpha_n,1-2a-\gamma_n\right) \text{.} \\
\end{align*}

The coordinates $\alpha_n$, $\beta_n$ and $\gamma_n$ are defined as
follows: 
\begin{align*}
\alpha_0&=a,\quad\beta_0=\frac 12+b,\quad\gamma_0=\frac 12-a+b\text{;}\\[3pt]
\alpha_{n+1} &=a+b-\gamma_n,\quad\text{if }a+b-\gamma_n\geq 0\text{,} \\
&=0,\quad\text{if }a+b-\gamma_n\leq 0\text{;}\\[3pt]
\beta_{n+1} &=a+b+\gamma_n,\quad\text{if }a+b+\gamma_n\leq 1\text{,} \\
&=1,\quad\text{if }a+b+\gamma_n\geq 1\text{;}\\[3pt]
\gamma_{n+1} &=-a+b+\beta_n,\quad\text{if }a+b\leq \frac 12\text{,} \\
&=1-a+b-\alpha_n,\quad\text{if }a+b\geq \frac 12\text{.}\\
\end{align*}

Notice that $\gamma_n\leq 1-a+b$, for every $n=0,1,2,\ldots $

\begin{proposition}
The tile is such that the region 
\begin{equation*}
\left[ 0,2a\right] \times \left[ 0,1-a+3b\right] \cup \left[ 2a,1\right]
\times \left[ 0,1-a+b\right] 
\end{equation*}
has the pattern $p4$ generated by the four transformations in the plane 
($\mathbb{R}^2$).
\end{proposition}

\begin{proof}
\textbf{1.} Assume that $a+b\leq \frac 12$ and $b>0$. Then 
\begin{align*}
\alpha_n &=2a-\left( n-1\right) b-\frac 12,\quad\text{for }1\leq n\leq 
   \frac{4a+2b-1}{2b}\text{,} \\
&=0,\quad\text{for }n\geq \frac{4a+2b-1}{2b}\text{;}\displaybreak[0]\\[3pt]
\beta_n &=\left( n+1\right) b+\frac 12,\quad\text{for }n\leq \frac{1-2b}{2b} \\
&=1,\quad\text{for }n\geq \frac{1-2b}{2b} \displaybreak[0]\\[3pt]
\gamma_n &=\left( n+1\right) b-a+\frac 12,\quad\text{for }n\leq \frac 1{2b} \\
&=1-a+b,\quad\text{for }n\geq \frac 1{2b}.
\end{align*}

Notice that\\
a) 
\begin{equation*}
4a+2b-1\leq 1-2b<1\text{;}
\end{equation*}
b) for every $n\geq 0$, 
\begin{equation*}
\left( a+b-\gamma_n\right) -\left( 2a-\beta_n\right) =b>0\text{;}
\end{equation*}
c) 
\begin{equation*}
2b-\gamma_n=b-bn+a-\frac 12\leq -bn\leq 0\text{;}
\end{equation*}
d) 
\begin{equation*}
\alpha_n-a-b=a-bn-\frac 12\leq -b\left( n+1\right) <0\text{;}
\end{equation*}
e) however, 
\begin{gather*}
\left( a+b+\gamma_n\right) -\left( 1-a+b\right) =a+b\left( n+1\right)
-\frac 12\text{,}\\
1-a+b-\beta_n=\frac 12-a-bn\text{;}
\end{gather*}
hence, if 
\begin{equation*}
\frac{1-2a-2b}{2b}<n<\frac{1-2a}{2b}\text{,}
\end{equation*}
there is a a small rectangle with vertices at $\left( 1-a+b,0\right) $, 
$\left( \sfrac 12+\left( n+2\right) b,0\right) $, $\left( \sfrac 12+\left(
n+2\right) b,\sfrac 12-a-nb\right) $ and $\left( 1-a+b,\sfrac 12-a-nb\right) $.

This rectangle is the translated by $u$ of the rectangle with vertices 
$( a+b,2b) $,
$( -\sfrac 12+( n+2) b+2a,2b)$, 
$( -\sfrac 12+( n+2) b+2a,\sfrac 12-a-nb+2b) $ and 
$( a+b,\sfrac 12-a-nb+2b) $.

\noindent f) As 
\begin{equation*}
\beta_n-\left( -\frac 12+\left( n+2\right) b+2a\right) =1-2a-b\geq 0\text{,}
\end{equation*}
g) and 
\begin{equation*}
\gamma_n-\left( \frac 12-a-nb+2b\right) =\left( 2n-1\right) b\geq 0\text{,}
\end{equation*}
this last rectangle is contained in $\mathcal{R}_n$, for $n\geq 1$.

\smallskip\noindent\textbf{2.} Assume that $a+b\geq \frac 12$ and $a<\frac 12$. Then 
\begin{align*}
\alpha_n &=\left( n+1\right) a-\frac n2,\quad\text{for }n\leq \frac{2a}{1-2a}
\\
&=0,\quad\text{for }n\geq \frac{2a}{1-2a} \\[3pt]
\beta_n &=2b-\left( n-1\right) a+\frac n2,\quad\text{for }1\leq n\leq 
    \frac{2-2a-4b}{1-2a} \\
&=1,\quad\text{for }n\geq \frac{2-2a-4b}{1-2a} \\[3pt]
\gamma_n &=-\left( n+1\right) a+b+\frac{n+1}2,\quad\text{for }n\leq \frac
1{1-2a} \\
&=1-a+b,\quad\text{for }n\geq \frac 1{1-2a}.
\end{align*}

Notice that\\
a) 
\begin{equation*}
2-2a-4b\leq 2a<1\text{;}
\end{equation*}
b) for every $n\leq \frac{2a}{1-2a}$
\begin{equation*}
\left( 1+2b-\alpha_n\right) -\left( a+b+\gamma_n\right) =\frac 12-a>0\text{;}
\end{equation*}
c) for every $n\leq \frac 1{1-2a}$
\begin{equation*}
1-2a-\gamma_n=\frac 12\left( 2a-1\right) \left( n-1\right) -b\leq 0\text{;}
\end{equation*}
d) for every $1\leq n\leq \frac{2-2a-4b}{1-2a}$
\begin{equation*}
a+b-\beta_n=\frac 12n\left( 2a-1\right) -b<0\text{;}
\end{equation*}
e) however, 
\begin{gather*}
\left( a+b-\gamma_n\right) -\left( a-b\right) =b+\frac 12\left( 2a-1\right)
\left( n+1\right) \text{,}\displaybreak[0]\\
-a+b+\alpha_n=b+\frac 12n\left( 2a-1\right) \text{;}
\end{gather*}
hence, if 
\begin{equation*}
\frac{2a+2b-1}{1-2a}<n<\frac{2b}{1-2a}\text{,}
\end{equation*}
there is a a small rectangle with vertices at $( a-\sfrac 12(
1-2a) ( n+1) ,0) $, $( a-b,0) $, $(
a-b,b+(\sfrac 12)n( 2a-1) ) $ and $( a-\sfrac 12(
1-2a) ( n+1) ,b+(\sfrac 12)n( 2a-1) ) $.

This rectangle is the rotation by $\sigma_2$ of the rectangle with vertices 
$( 1-a-b,0) $, 
$( 1-a-(\sfrac 12)n( 1-2a) ,0) $, 
$( 1-a-(\sfrac 12)n( 1-2a) ,-b+\sfrac 12( 1-2a)
    ( n+1) ) $ 
and 
$( 1-a-b,-b+\sfrac 12( 1-2a) ( n+1) ) $.

As\\
f) 
\begin{equation*}
\left( 1-a+b-\alpha_n\right) -\left( -b+\frac 12\left( 1-2a\right) \left(
n+1\right) \right) \\
= \frac 12-a+2b\geq 0\text{,}
\end{equation*}
g) 
\begin{equation*}
\begin{split}
\left( a+b+\gamma_n\right) -\left( 1-a-\frac 12n\left( 1-2a\right) \right)
&= a+2b-\frac 12+n\left( 1-2a\right) \\
&> 3a+4b-\frac 32\geq b\text{;}
\end{split}
\end{equation*}
this last rectangle is contained in the red rectangle with vertex $P_4$.

\smallskip\noindent\textbf{3.} Assume that $a<\sfrac 12$ and $b=0$.

In this case $u=\left( 1-2a,0\right) $. The generating region, \textbf{gr},
of Figure \ref{fig:22} 
is the square with vertices at $\left( a,0\right) $, $\left(
\sfrac 12,0\right) $, $\left( \sfrac 12,\sfrac 12-a\right) $ and $\left(
a,\sfrac 12-a\right) $. Rotating it around the center $\left( \sfrac 12,\sfrac
12-a\right) $, one obtains another square with vertices at $\left(
a,0\right) $, $\left( 1-a,0\right) $, $\left( 1-a,1-2a\right) $ and $\left(
a,1-2a\right) $. Translating this last square by $pu$, with $p\in \mathbb{Z}$,
one obtains $\left[ 0,1\right] \times \left[ 0,1-2a\right] $. Rotating this
region around the center $\left( \sfrac 12,\sfrac 12-a\right) $, one obtains
the union of $\left[ 0,1\right] \times \left[ 0,1-2a\right] $ with $\left[
a,1-a\right] \times \left[ 0,1-a\right] $. Translating it by $pu$, with 
$p\in \mathbb{Z}$, one obtains $\left[ 0,1\right] \times \left[ 0,1-a\right] $.

\smallskip\noindent\textbf{4.} Assume that $b>0$ and $a=\sfrac 12$.

In this case $u=\left( 0,-2b\right) $. The generating region, \textbf{gr},
of Figure \ref{fig:22} 
is the square with vertices at $\left( \sfrac 12,0\right) $, 
$\left( \sfrac 12+b,0\right) $, $\left( \sfrac 12+b,b\right) $ and $\left(
\sfrac 12,b\right) $. Rotating it around the center $\left( \sfrac 12,b\right) 
$, one obtains another square with vertices at $\left( \sfrac 12-b,0\right) $, 
$\left( \sfrac 12+b,0\right) $, $\left( \sfrac 12+b,2b\right) $ and $\left(
\sfrac 12-b,2b\right) $. Translating this last square by $pu$, with $p\in 
\mathbb{Z}$, one obtains $\left[ \sfrac 12-b,\sfrac 12+b\right] \times \left[
0,1\right] $. Rotating this region around the center $\left( \sfrac
12,b\right) $, one obtains the union of $\left[ \sfrac 12-b,\sfrac 12+b\right]
\times \left[ 0,1\right] $ with $\left[ 0,1\right] \times \left[ 0,2b\right] 
$. Translating it by $pu$, with $p\in \mathbb{Z}$, one obtains $\left[
0,1\right] \times \left[ 0,1\right] $.
\end{proof}

\subsubsection{Case H}

In this case the transformations are a translation by  $u=(1-a+b,1-a-b)$, 
two rotations of order $4$ ($\sigma_0$ and $\sigma_1$) 
with centers located at $\left( a,b\right) $ and $\left( a+b,1-a+b\right) $, 
and a rotation of order $2$ ($\sigma_2$) with center located at $\bigl( 
\frac{1+a+b}2,\frac{1-a+b}2\bigr) $. These transformations are compatible
in the sense that they generate a pattern $p4$ in the plane.

\begin{figure}[htb!]
  \centering
  \includegraphics[width=4.7444in]{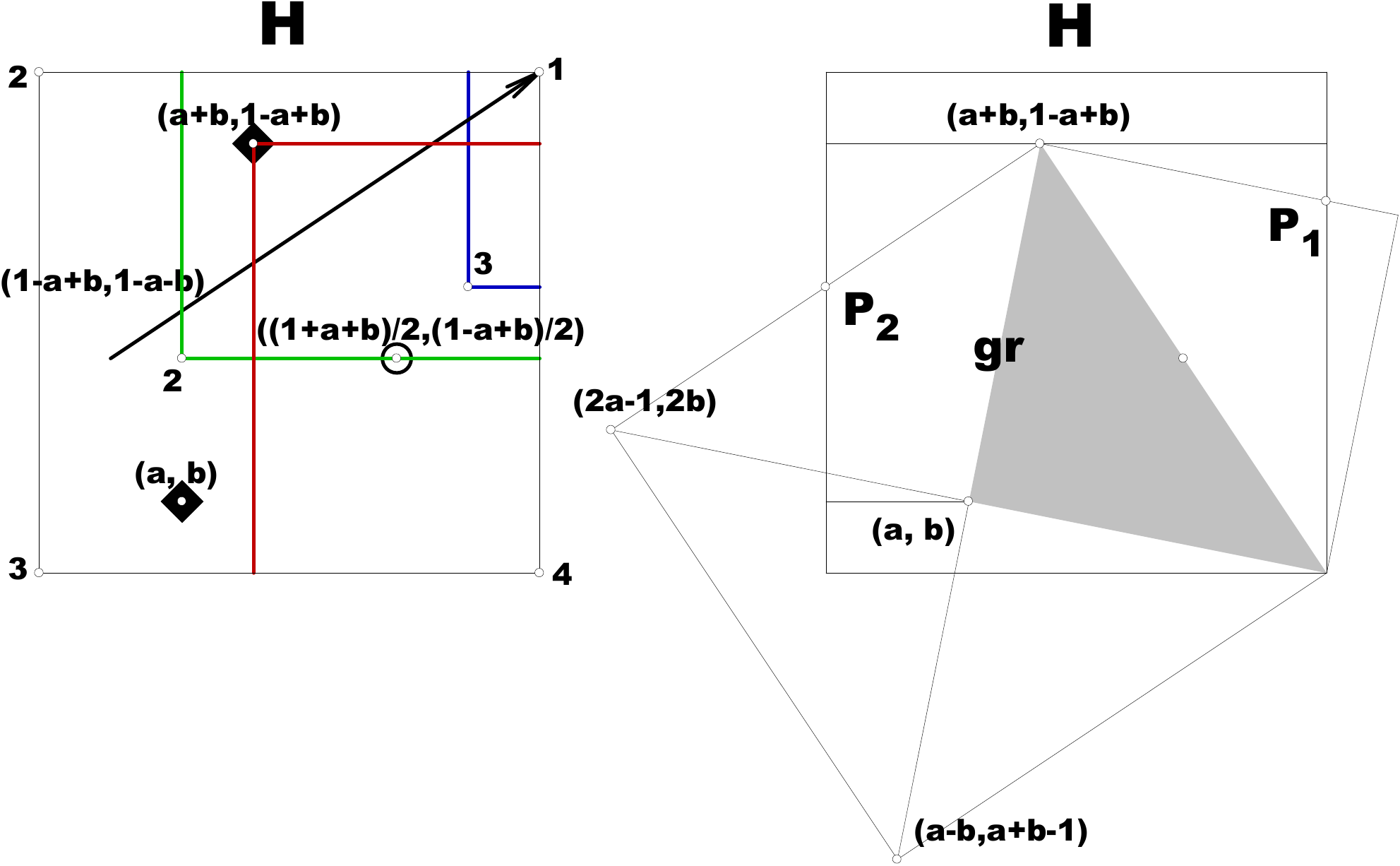}
  \vspace*{-8pt}
  \caption{}
  \label{fig:24}
\end{figure}

\begin{proposition}
The tile is such that it has the pattern $p4$ generated by the four
transformations in the plane ($\mathbb{R}^2$).
\end{proposition}

\begin{proof}
Consider the dashed triangle \textbf{gr} contained in the tile 
(Figure \ref{fig:24}); 
\textbf{gr} is a generating region of the pattern $p4$ in the plane. Its
rotation by $\sigma_0$, $\sigma_0^2$ and $\sigma_0^3$ covers the union of 
$\left[ a,1\right] \times \left[ 0,b\right] $, the quadrilateral with
vertices $\left( 0,b\right) $, $\left( a,b\right) $, $\left(
a+b,1-a+b\right) $ and $P_2$; $P_2=\left( 0,2b+\sfrac{\left( 1-2a\right)
\left( 1-a-b\right) }{(1-a+b)}\right) $. Its rotation by $\sigma_2$ covers
the triangle with vertices $\left( a+b,1-a+b\right) $, $\left( 1,0\right) $
and $P_1$; $P_1=\left( 1,1-a+\sfrac{b^2}{(1-a)}\right) $. The triangle with
vertices $\left( a+b,1-a+b\right) $, $P_1$ and $\left( 1,1-a+b\right) $, is
the rotation by $\sigma_1$ of a triangle in \textbf{gr}. The triangle with
vertices $\left( 0,1-a+b\right) $, $P_2$ and $\left( a+b,1-a+b\right) $, is
the rotation by $\sigma_1^{-1}$ of a triangle in \textbf{gr}.

Hence, the region union of $\left[ a,1\right] \times \left[ 0,b\right] $ and 
$\left[ 0,1\right] \times \left[ b,1-a+b\right] $ is covered. Its rotation
by $\sigma_0^2$ covers $\left[ 0,a\right] \times \left[ 0,b\right] $. The
rotation of $\left[ 0,1\right] \times \left[ 0,1-a+b\right] $ by $\sigma_1$
and $\sigma_1^{-1}$ covers $\left[ 0,1\right] \times \left[ 1-a+b,1\right] $. 
Therefore the proposition is proved.
\end{proof}

\subsubsection{Periodicity of the tile pattern}

Let us briefly point out that all the transformations involving cases A, B,
C, D, E, F, G and H, transform the generating regions (\textbf{gr}, see
Figures \ref{fig:20}--\ref{fig:24}), covering the tile.

Remember now that
\begin{align*}
&\begin{tabular}{@{}ccccc@{}}
$i\rightarrow $ & A & B & C & D \\ 
\midrule
$u_i$ & $\left( \frac 12,\frac 12\right) $ & $\left( 1-b,a\right) $ & 
$\left( \frac 12-a-b,\frac 12+a-b\right) $ & $\left( a,b\right) $ \\ 
$v_i$ & $\left( -\frac 12,\frac 12\right) $ & $\left( -a,1-b\right) $ & 
$\left( -\frac 12-a+b,\frac 12-a-b\right) $ & $\left( -b,a\right) $ \\ 
\end{tabular}\displaybreak[0]\\[5pt]
&\begin{tabular}{@{}ccccc@{}}
$i\rightarrow $ & E & F & G & H \\ 
\midrule
$u_i$ & $\left( \frac 12,\frac 12\right) $ & $\left( a,b\right) $ & $\left(
\frac 12-a+b,\frac 12-a-b\right) $ & $\left( b,1-a\right) $ \\ 
$v_i$ & $\left( -\frac 12,\frac 12\right) $ & $\left( -b,a\right) $ & 
$\left( a+b-\frac 12,\frac 12-a+b\right) $ & $\left( a-1,b\right) $ \\ 
\end{tabular}
\end{align*}

\smallskip\noindent\textbf{1.} For all the cases involving the three
tiles $\left[ 0,1\right]^2$, $\left[ -1,0\right] \times \left[
  0,1\right] $ and $\left[ 0,1\right] \times \left[ -1,0\right] $ one
has the following:
\begin{compactenum}[ a)]
\item For the seven cases AE, AF, AG, AH, BE, CE and DE, one has 
\begin{equation*}
pu_A+qv_A=\frac 12\left( p-q,p+q\right) =p\left( \frac 12,\frac 12\right)
+q\left( -\frac 12,\frac 12\right)
\end{equation*}
or 
\begin{equation*}
pu_E+qv_E=\frac 12\left( p-q,p+q\right) =p\left( \frac 12,\frac 12\right)
+q\left( -\frac 12,\frac 12\right) \text{;}
\end{equation*}
\item BF 
\begin{equation*}
pu_B+qv_B+qu_F-pv_F=\left( p,q\right) =p\left( 1,0\right) +q\left(
0,1\right) \text{;}
\end{equation*}
\item BG 
\begin{equation*}
\left( p+q\right) u_B+\left( q-p\right) v_B+pu_G+qv_G=q\left( \frac 12,\frac
32\right) -p\left( -\frac 32,\frac 12\right) \text{;}
\end{equation*}
\item BH 
\begin{equation*}
pu_B+qv_B+pu_H+qv_H=p\left( 1,1\right) +q\left( -1,1\right) \text{;}
\end{equation*}
\item CF 
\begin{equation*}
pu_C+qv_C+\left( p+q\right) u_F+\left( q-p\right) v_F=p\left( \frac 12,\frac
12\right) +q\left( -\frac 12,\frac 12\right) \text{;}
\end{equation*}
\item CG 
\begin{equation*}
pu_C-qv_C+qu_G+pv_G=q\left( 1,0\right) +p\left( 0,1\right) \text{;}
\end{equation*}
\item CH 
\begin{equation*}
pu_C-qv_C+\left( p+q\right) u_H+\left( p-q\right) v_H=q\left( \frac 32,\frac
12\right) +p\left( -\frac 12,\frac 32\right) \text{;}
\end{equation*}
\item DF 
\begin{equation*}
pu_D+qv_D-qv_F-pu_F=\left( 0,0\right) \text{;}
\end{equation*}
\item DG 
\begin{equation*}
\left( p+q\right) u_D+\left( p-q\right) v_D+pu_G-qv_G=p\left( \frac 12,\frac
12\right) -q\left( -\frac 12,\frac 12\right) \text{;}
\end{equation*}
\item  DH 
\begin{equation*}
pu_D+qv_D+qu_H-pv_H=p\left( 1,0\right) +q\left( 0,1\right) \text{.}
\end{equation*}
\end{compactenum}

Hence, for the nine cases AE, AF, AG, AH, BE, CE, DE, CF and DG the tile has
a pattern with rotation centers of order $4$ at: 
\begin{equation*}
\left( a,b\right) +p\left( \frac 12,\frac 12\right) +q\left( -\frac 12,\frac
12\right) \text{;}
\end{equation*}
for the cases BF, CG, and DH: 
\begin{equation*}
\left( a,b\right) +p\left( 1,0\right) +q\left( 0,1\right) \text{;}
\end{equation*}
for the case BH: 
\begin{equation*}
\left( a,b\right) +p\left( 1,1\right) +q\left( -1,1\right) \text{;}
\end{equation*}
for the case BG: 
\begin{equation*}
\left( a,b\right) +p\left( \frac 12,\frac 32\right) +q\left( -\frac 32,\frac
12\right) \text{;}
\end{equation*}
for the case CH: 
\begin{equation*}
\left( a,b\right) +p\left( \frac 32,\frac 12\right) +q\left( -\frac 12,\frac
32\right) \text{;}
\end{equation*}
with $p,q\in \mathbb{Z}$; for the case DF there is no conclusion.

\smallskip\noindent\textbf{2.} Remember that
\begin{equation*}
\begin{tabular}{@{}ccccc@{}}
$i\rightarrow $ & $\left( c_1=\right) 1$ & $\left( c_1=\right) 2$ & $\left(
c_1=\right) 3$ & $\left( c_1=\right) 4$ \\ 
\midrule
$u_i$ & $\left( 1,0\right) $ & $\left( \frac 12-b,\frac 12+a\right) $ & 
$\left( a-b,a+b\right) $ & $\left( \frac 12+b,\frac 12-a\right) $ \\ 
$v_i$ & $\left( 0,1\right) $ & $\left( -\frac 12-a,\frac 12-b\right) $ & 
$\left( -a-b,a-b\right) $ & $\left( a-\frac 12,\frac 12+b\right) $ \\ 
\end{tabular}
\end{equation*}

In the following the notations BG$i$, CH$i$, DF$i$, mean that $c_1=i$; also 
$u=u_D=u_F$ and $v=v_D=v_F$ (see Figures \ref{fig:16} and \ref{fig:17}).

For all the cases involving the three tiles $\left[ 0,1\right]^2$, $\left[
-1,0\right] \times \left[ 0,1\right] $, $\left[ 0,1\right] \times \left[
-1,0\right] $ and $\left[ -1,0\right] \times \left[ -1,0\right] $ one has
the following:
\begin{compactenum}[ a)]
\item For the three cases BG$1$, CH$1$ and DF$1$, one has 
\begin{equation*}
pu_1+qv_1=\left( p,q\right) =p\left( 1,0\right) +q\left( 0,1\right) \text{;}
\end{equation*}
\item BG2 
\begin{gather*}
pu_B+qv_B-pu_2-qv_2=q\left( \frac 12,\frac 12\right) +p\left( \frac
12,-\frac 12\right) \text{;}\displaybreak[0]\\
pu_G+qv_G+\left( p+q\right) u_2+\left( q-p\right) v_2=p\left( \frac 32,\frac
12\right) +q\left( -\frac 12,\frac 32\right) \text{;}
\end{gather*}
\item BG3 
\begin{gather*}
\left( q-p\right) u_B+\left( p+q\right) v_B+pu_3-qv_3=q\left( 1,1\right)
+p\left( -1,1\right) \text{;}\displaybreak[0]\\
pu_G+qv_G+pu_3+qv_3=p\left( \frac 12,\frac 12\right) +q\left( -\frac
12,\frac 12\right) \text{;}
\end{gather*}
\item BG4 
\begin{gather*}
pu_B+qv_B+pu_4+qv_4=p\left( \frac 32,\frac 12\right) +q\left( -\frac
12,\frac 32\right) \text{;}\displaybreak[0]\\
pu_G+qv_G-\left( p+q\right) u_4+\left( p-q\right) v_4=p\left( -\frac
12,\frac 12\right) -q\left( \frac 12,\frac 12\right) \text{;}
\end{gather*}
\item CH2 
\begin{gather*}
pu_H+qv_H+pu_2+qv_2=p\left( \frac 12,\frac 32\right) +q\left( -\frac
32,\frac 12\right) \text{;}\displaybreak[0]\\
-pu_C+qv_C+\left( p+q\right) u_2+\left( p-q\right) v_2=q\left( \frac
12,\frac 12\right) +p\left( -\frac 12,\frac 12\right) \text{;}
\end{gather*}
\item CH3 
\begin{gather*}
\left( p+q\right) u_H+\left( -p+q\right) v_H+pu_3+qv_3=p\left( 1,1\right)
+q\left( -1,1\right) \text{;}\displaybreak[0]\\
pu_C+qv_C+qu_3-pv_3=p\left( \frac 12,\frac 12\right) +q\left( -\frac
12,\frac 12\right) \text{;}
\end{gather*}
\item CH4 
\begin{gather*}
pu_H-qv_H-pu_4+qv_4=q\left( \frac 12,\frac 12\right) +p\left( -\frac
12,\frac 12\right) \text{;}\displaybreak[0]\\
pu_C+qv_C+\left( p-q\right) u_4+\left( p+q\right) v_4=p\left( \frac 12,\frac
32\right) +q\left( -\frac 32,\frac 12\right) \text{;}
\end{gather*}
\item DF2 
\begin{equation*}
pu-qv+qu_2+pv_2=q\left( \frac 12,\frac 12\right) +p\left( -\frac 12,\frac
12\right) \text{;}
\end{equation*}
\item DF3 
\begin{equation*}
\left( q-p\right) u-\left( p+q\right) v+pu_3+qv_3=\left( 0,0\right) \text{;}
\end{equation*}
\item DF4 
\begin{equation*}
-pu+qv+qu_4+pv_4=q\left( \frac 12,\frac 12\right) +p\left( -\frac 12,\frac
12\right) \text{.}
\end{equation*}
\end{compactenum}

Hence, for the eleven cases BG1, CH1, DF1, BG2, BG3, BG4, CH2, CH3, CH4, DF2
and DF4 the tile has a pattern with rotation centers of order $4$ at: 
\begin{equation*}
\left( a,b\right) +p\left( \frac 12,\frac 12\right) +q\left( -\frac 12,\frac
12\right) \text{;}
\end{equation*}
with $p,q\in \mathbb{Z}$; for the case DF3 there is no conclusion.

\smallskip\noindent\textbf{3.} In the following the notation DF$3i$,
means that $c_1=3$ and $d_1=i$; also, $u=u_D=u_F$, $v=v_D=v_F$,
$u_i=u_{DF3i}$ and $v_i=v_{DF3i}$ (see Figure \ref{fig:17}).

\begin{equation*}
\begin{tabular}{@{}ccccc@{}}
\toprule
& DF31 & DF32 & DF33 & DF34 \\ 
\cmidrule(l){2-5}
$4$ & 
\cellboxt{$(a+b-\sfrac 12,\\ \xqquad\sfrac 12-a+b)$} & 
$\left( 1-b,a\right) $ & 
\cellboxt{$(\sfrac 12-a+b,\\ \xqquad\sfrac 32-a-b)$} & 
$\left( 1-a,1-b\right) $ \\ \addlinespace
$4$ & 
$( a-\sfrac 12,\sfrac 12+b) $ & 
$\left( 1-a,-b\right) $ & 
$\left( \sfrac 12+b,\sfrac 32-a\right) $ & 
$\left( 1+b,1-a\right) $ \\ \addlinespace
$4$ & 
- & 
\cellboxt{$( 1-a-b,\\ \xqqquad a-b)$} & 
- & 
\cellboxt{$( 1-a+b,\\ \xqquad 1-a-b)$} \\ \addlinespace
$2$ & 
- & 
- & 
$\Bigl( 1-\frac{a+b}2,\frac{1+a-b}2\Bigr) $ & 
$\left(b,1-a\right) $ \\ \addlinespace
$2$ & 
- & 
- & 
$\Bigl( 1+\frac{b-a}2,\frac{1-a-b}2\Bigr) $ & 
$\Bigl( \frac{a+b}2,1+\frac{b-a}2\Bigr) $ \\ \addlinespace
2 & 
- & 
- & 
$\Bigl( 1-a,\frac 12-b\Bigr) $ & 
- \\ \addlinespace
tr & 
\cellboxt{$( 1-a-b,\\ \xqqquad a-b)$} & 
$\left( 2b-1,1-2a\right) $ & 
- & 
- \\ \addlinespace

tr & 
\cellboxt{$( a-b-1,\\ \xqqquad a+b)$} & 
\cellboxt{$( a+b-1,\\ \xqquad 1-a+b)$} & 
- & 
- \\ \addlinespace
tr & 
$\left( 2a-1,2b\right) $ & 
- & 
- & 
- \\ 
\bottomrule
\end{tabular}
\end{equation*}

\begin{equation*}
\begin{tabular}{@{}ccccc@{}}
$i\rightarrow $ & 
$\left( d_1=\right) 1$ & 
$\left( d_1=\right) 2$ & 
$\left(d_1=\right) 3$ & 
$\left( d_1=\right) 4$ \\ 
\midrule
$u_i$ & 
\cellboxt{$(\sfrac 12-a+b,\\ \xquad\sfrac 12-a-b)$} & 
\cellboxt{$(1-2a-b,\\ \xquad a-2b)$} & 
\cellboxt{$(\sfrac 12-2a+2b,\\ \xquad\sfrac 32-2a-2b)$} & 
\cellboxt{$(1-2a+b,\\ \xquad 1-a-2b)$} \\ \addlinespace
$v_i$ & 
\cellboxt{$(a+b-\sfrac 12,\\ \xquad\sfrac 12-a+b)$} & 
\cellboxt{$(2b-a,\\ \xquad 1-2a-b)$} & 
\cellboxt{$(2a+2b-\sfrac 32,\\ \xquad\sfrac 12-2a+2b)$} & 
\cellboxt{$(a+2b-1,\\ \xquad 1-2a+b)$} \\ 
\end{tabular}
\end{equation*}

Notice that
\begin{gather*}
v_1-u_1=\left( 2a-1,2b\right),\quad
u_2=\left( 1-a-b,a-b\right) -\left( a,b\right) , \\
\left( 1-a,\tfrac 12-b\right) -\left( a,b\right)
  =\left( 1-2a,\tfrac 12-2b\right) , \\
u_3=\left( 1-2a,\tfrac 12-2b\right) +\left( 2b-\tfrac 12,1-2a\right) , \\
u_4=\left( 1-a+b,1-a-b\right) -\left(a,b\right) .
\end{gather*}

\begin{compactenum}[ a)]
\item DF31
\begin{equation*}
\left( p-q\right) u+\left( p+q\right) v+pu_1+qv_1=p\left( \frac 12,\frac
12\right) +q\left( -\frac 12,\frac 12\right) \text{;}
\end{equation*}
\item DF32
\begin{equation*}
\left( 2p+q\right) u+\left( 2q-p\right) v+pu_2+qv_2=p\left( 1,0\right)
+q\left( 0,1\right) \text{;}
\end{equation*}
\item DF33
\begin{equation*}
\left( 2p-2q\right) u+\left( 2p+2q\right) v+pu_3+qv_3=p\left( \frac 12,\frac
32\right) +q\left( -\frac 32,\frac 12\right) \text{;}
\end{equation*}
\item DF34
\begin{equation*}
\left( 2p-q\right) u+\left( p+2q\right) v+pu_4+qv_4=p\left( 1,1\right)
+q\left( -1,1\right) \text{.}
\end{equation*}
\end{compactenum}

Hence, for the case DF31, the tile has a pattern with rotation centers of
order $4$ at: 
\begin{equation*}
\left( a,b\right) +p\left( \frac 12,\frac 12\right) +q\left( -\frac 12,\frac
12\right) \text{;}
\end{equation*}
for the case DF32, the tile has a pattern with rotation centers of order $4$
at: 
\begin{equation*}
\left( a,b\right) +p\left( 1,0\right) +q\left( 0,1\right) \text{;}
\end{equation*}
for the case DF34, the tile has a pattern with rotation centers of order $4$
at: 
\begin{equation*}
\left( a,b\right) +p\left( 1,1\right) +q\left( -1,1\right) \text{;}
\end{equation*}
with $p,q\in \mathbb{Z}$; for the case DF33 there is no conclusion.

\smallskip\noindent\textbf{4.} In the following the notation DF$33i$, means
that $c_1=3$, $d_1=3$ and 
$e_i=i $; also, $u=u_D=u_F$, $v=v_D=v_F$, $u_i=u_{DF33i}$ and $v_i=v_{DF33i}$
(see Figure \ref{fig:17}).

\begin{equation*}
\begin{tabular}{@{}ccccc@{}}
\toprule
& DF331 & DF332 & DF333 & DF334 \\
\cmidrule(l){2-5}
$4$ & 
\cellboxt{$( a+b,\\ \xquad 1-a+b)$} & 
$\left( \frac 12-b,a-\frac 12\right) $ & 
\cellboxt{$(1-a+b,\\ \xquad 1-a-b)$} & 
$\left( \frac 32-a,\frac 12-b\right) $ \\ \addlinespace 
$4$ & 
$\left( a,1+b\right) $ & 
- & 
$\left( 1+b,1-a\right) $ & 
- \\ \addlinespace 
$4$ & 
$\left( b,1-a\right) $ & 
- & 
$\left( 1-a,1-b\right) $ & 
- \\ \addlinespace 
$2$ & 
- & 
- & 
$\left( 1-\frac{a+b}2,\frac{a-b}2\right) $ & 
$\left( \frac12+b,1-a\right) $ \\ \addlinespace 
$2$ & 
- & 
- & 
- & 
$\left( \frac{1+a+b}2,1+\frac{b-a}2\right) $ \\ \addlinespace 
$2$ & 
- & 
- & 
- & 
$\left( \frac{1-a+b}2,1-\frac{a+b}2\right) $ \\ \addlinespace 
tr & 
\cellboxt{$(a+b-1,\\ \xquad 1-a+b) $} & 
$\left( 2b,1-2a\right) $ & 
- & 
- \\ \addlinespace 
tr & 
- & 
$\left( a+b,1-a+b\right) $ & 
- & 
- \\ \addlinespace 
tr & 
- & 
$\left( b-a,1-a-b\right) $ & 
- & 
- \\
\bottomrule 
\end{tabular}
\end{equation*}

\begin{tabular}{@{}ccccc@{}}
$i\rightarrow $ & 
$( e_1=) 1$ & 
$( e_1=) 2$ & 
$(e_1=) 3$ & 
$( e_1=) 4$ \\ 
\midrule
$u_i$ & 
$( b,1-a) $ & 
\cellboxt{$( \sfrac 12-a+b,\\ \xqquad\sfrac 12-a-b) $} & 
\cellboxt{$( 1-2a+b,\\ \xqquad 1-a-2b) $} & 
$( \sfrac 32-2a,\sfrac 12-2b) $ \\ \addlinespace
$v_i$ & 
$( a-1,b) $ & 
\cellboxt{$( a+b-\sfrac 12,\\ \xqquad\sfrac 12-a+b) $} & 
\cellboxt{$( a+2b-1,\\ \xqquad 1-2a+b) $} & 
$( 2b-\sfrac 12,\sfrac 32-2a) $ \\ 
\end{tabular}

Notice that
\begin{gather*}
u_1=\left( a+b,1-a+b\right) -\left( a,b\right) , \quad
u_2+v_2=\left( 2b,1-2a\right) ,\\
u_3=\left( 1-a+b,1-a-b\right) -\left(a,b\right) ,\\
\left( \tfrac 12+b,1-a\right) -\left( a,b\right) =\left( \tfrac
12-a+b,1-a-b\right) ,\\
u_4=\left( \tfrac 12-a+b,1-a-b\right) +\left(
1-a-b,a-b-\tfrac 12\right) .
\end{gather*}

\begin{compactenum}[ a)]
\item DF331 
\begin{equation*}
pu+qv+qu_1-pv_1=p\left( 1,0\right) +q\left( 0,1\right) \text{;}
\end{equation*}
\item DF332 
\begin{equation*}
\left( p-q\right) u+\left( p+q\right) v+pu_2+qv_2=p\left( \frac 12,\frac
12\right) +q\left( -\frac 12,\frac 12\right) \text{;}
\end{equation*}
\item DF333 
\begin{equation*}
\left( 2p-q\right) u+\left( p+2q\right) v+pu_3+qv_3=p\left( 1,1\right)
+\left( -1,1\right) \text{;}
\end{equation*}
\item DF334 
\begin{equation*}
2pu+2qv+pu_4+qv_4=p\left( \frac 32,\frac 12\right) +q\left( -\frac 12,\frac
32\right) \text{;}
\end{equation*}
\end{compactenum}

\begin{equation*}
-pu_{DF33}-qv_{DF33}+\left( p-q\right) u_4+\left( p+q\right) v_4=p\left(
\frac 12,\frac 12\right) +q\left( -\frac 12,\frac 12\right) \text{.}
\end{equation*}

Hence, for the cases DF332 and DF334, the tile has a pattern with rotation
centers of order $4$ at: 
\begin{equation*}
\left( a,b\right) +p\left( \frac 12,\frac 12\right) +q\left( -\frac 12,\frac
12\right) \text{;}
\end{equation*}
for the case DF331, the tile has a pattern with rotation centers of order $4$
at: 
\begin{equation*}
\left( a,b\right) +p\left( 1,0\right) +q\left( 0,1\right) \text{;}
\end{equation*}
for the case DF333, the tile has a pattern with rotation centers of order $4$
at: 
\begin{equation*}
\left( a,b\right) +p\left( 1,1\right) +q\left( -1,1\right) \text{;}
\end{equation*}
with $p,q\in \mathbb{Z}$.

Hence, we have proved

\begin{proposition}
Let $\left( a,b\right) $ the rotation center of order $4$ with $0\leq b\leq
a\leq \sfrac 12$, $a>0$, $\left( a,b\right) \neq \left( \sfrac 12,0\right) $.
Consider the tiles $\left[ 0,1\right]^2$ and $\left[ 0,1\right] \times
\left[ -1,0\right] $, and their four possibilities: E, F, G and H. The tile
has a pattern $p_4$ with the following possible translation invariances:
\begin{compactenum}[ 1)]
\item F, G and H: $p\left( 1,0\right) +q\left( 0,1\right) $, \; $p,q\in \mathbb{Z}$.
\item F and H: $p\left( 1,1\right) +q\left( -1,1\right) $, \; $p,q\in \mathbb{Z}$.
\item E, F, G and H: $p\left( \frac 12,\frac 12\right) +q\left( -\frac 12,\frac
12\right) $, \; $p,q\in \mathbb{Z}$.
\end{compactenum}
\end{proposition}

\subsubsection{Mathematical classification of tiles in the general case}

Consider a tile in $\mathbb{R}^2$ like it is shown in Figure \ref{fig:07}. 
Let $n$ the
number of rotation centers of order $4$ in the tile. Assume that there are
rotation centers of order $4$ located at $\left( a,b\right) $ and $\left(
a+\alpha ,b+\beta \right) $, with $\alpha >0$ and $\beta \geq 0$, and that,
in these conditions, $\left( a+\alpha ,b+\beta \right) $ is the nearest from 
$\left( a,b\right) $. Then, the rotation centers of order $4$ are located at 
\begin{equation*}
\left( a,b\right) +r\left( \alpha ,\beta \right) +s\left( -\beta ,\alpha
\right) \text{,}
\end{equation*}
with $r,s\in \mathbb{Z}$. The number of rotation centers in the tile is 
$n=\frac 1{\alpha^2+\beta^2}$.

We consider three possibilities.

\smallskip\noindent\textbf{1}. There is a rotation center at $\left(
  a+1,b\right) $ (as already seen this can be the situation in the
cases F, G and H; we call them F1, G1 and H1). Then, there are $p,q\in
\mathbb{Z}$ such that
\begin{equation*}
p\left( \alpha ,\beta \right) +q\left( -\beta ,\alpha \right) =\left(
1,0\right) \text{;}
\end{equation*}
hence 
\begin{equation*}
\alpha =\frac p{p^2+q^2}\text{, }\beta =\frac{-q}{p^2+q^2}\text{,}
\end{equation*}
$n=p^2+q^2$ and the centers are located at the points. 
\begin{equation*}
\left( a,b\right) +\frac 1{p^2+q^2}\left( rp+sq,sp-rq\right) \text{.}
\end{equation*}

\smallskip\noindent\textbf{2}. There is a rotation center at $\left( a+1,b+1\right) $ (as
already seen this can be the situation in the cases F and H; we call them F2
and H2). Then, there are $p,q\in \mathbb{Z}$ such that 
\begin{equation*}
p\left( \alpha ,\beta \right) +q\left( -\beta ,\alpha \right) =\left(
1,1\right) \text{;}
\end{equation*}
hence 
\begin{equation*}
\alpha =\frac{p+q}{p^2+q^2}\text{, }\beta =\frac{p-q}{p^2+q^2}\text{,}
\end{equation*}
$n=\frac{p^2+q^2}2$ and the centers are located at the points 
\begin{equation*}
\left( a,b\right) +\frac 1{p^2+q^2}\left( r\left( p+q\right) +s\left(
q-p\right) ,s\left( p+q\right) +r\left( p-q\right) \right) \text{.}
\end{equation*}

\smallskip\noindent\textbf{3}. There is a rotation center at $\left( a+\frac 12,b+\frac
12\right) $ (as already seen this can be the situation in the cases E, F, G
and H; we call them E3, F3, G3 and H3). Then, there are $p,q\in \mathbb{Z}$
such that 
\begin{equation*}
p\left( \alpha ,\beta \right) +q\left( -\beta ,\alpha \right) =\left( \frac
12,\frac 12\right) \text{;}
\end{equation*}
hence 
\begin{equation*}
\alpha =\frac{p+q}{2\left( p^2+q^2\right) }\text{, }
\beta =\frac{p-q}{2\left( p^2+q^2\right) }\text{,}
\end{equation*}
$n=2\left( p^2+q^2\right) $ and the centers are located at the points 
\begin{equation*}
\left( a,b\right) +\frac 1{2\left( p^2+q^2\right) }\left( r\left( p+q\right)
+s\left( q-p\right) ,s\left( p+q\right) +r\left( p-q\right) \right) \text{.}
\end{equation*}

\textbf{These are always in the first type of general tiles}. The cases E3,
F3, G3 and H3 are in the first type of general tiles.

Hence, let us now consider the other five cases: F1, F2, G1, H1, H2.

Remember that 
\begin{equation*}
\begin{tabular}{@{}ccccc@{}}
$i\rightarrow $ & E & F & G & H \\ 
\midrule
$u_i$ & $\left( \frac 12,\frac 12\right) $ & $\left( a,b\right) $ & $\left(
\frac 12-a+b,\frac 12-a-b\right) $ & $\left( b,1-a\right) $ \\ \addlinespace
$v_i$ & $\left( -\frac 12,\frac 12\right) $ & $\left( -b,a\right) $ & 
$\left( a+b-\frac 12,\frac 12-a+b\right) $ & $\left( a-1,b\right) $ \\ 
\end{tabular}
\end{equation*}
and let $\left( x,y\right) $ the coordinates of the rotation centers of
order $4$. In the following we use the notations 
$p,q,p_i,q_i,r,s,r_{i,}s_i\in \mathbb{Z}$, $i=0,1,\ldots $

\smallskip\noindent\textbf{F1.} In this case
\begin{align*}
\left( x,y\right) &= \left( a,b\right) +\frac 1{p^2+q^2}\left(
rp+sq,sp-rq\right) \text{,} \\
&\qquad \left( a,b\right) 
    =\left( \frac{r_0p+s_0q}{p^2+q^2},\frac{s_0p-r_0q}{p^2+q^2}\right) \\
\left( x,y\right) &=\frac 1{p^2+q^2}\left( \left( r+r_0\right) p+\left(
s+s_0\right) q,\left( s+s_0\right) p-\left( r+r_0\right) q\right) \text{,} \\
\left( x,y\right) &=\frac 1{p^2+q^2}\left( r_1p+s_1q,s_1p-r_1q\right) \text{.}
\end{align*}
These are included in the first and second types of general tiles.

\smallskip\noindent\textbf{F2.} In this case
\begin{align*}
\left( x,y\right) &=\left( a,b\right) +\frac 1{p^2+q^2}\left( r\left(
p+q\right) +s\left( q-p\right) ,s\left( p+q\right) +r\left( p-q\right)
\right) \text{,} \\
\left( x,y\right) &=\frac 1{p^2+q^2}\bigl( \left( r+r_0\right) \left(
p+q\right) +\left( s+s_0\right) \left( q-p\right),\\
&\qquad\qquad\qquad\left( s+s_0\right)
\left( p+q\right) +\left( r+r_0\right) \left( p-q\right) \bigr) \text{,}\\
\left( x,y\right) &=\frac 1{p^2+q^2}\left( r_1\left( p+q\right) +s_1\left(
q-p\right) ,s_1\left( p+q\right) +r_1\left( p-q\right) \right) \text{.}
\end{align*}
Let $p_1=p+q$ and $q_1=q-p$. Then $p=\frac{p_1-q_1}2$ and $q=\frac{p_1+q_1}2$
and 
\begin{equation*}
\left( x,y\right) =\frac 2{p_1^2+q_1^2}\left(
r_1p_1+s_1q_1,s_1p_1-r_1q_1\right) \text{.}
\end{equation*}
\begin{compactenum}[ a)]
\item If $p+q$ is odd, then $p_1$ and $q_1$ are odd numbers and these are the
fourth type of general tiles.
\item If $p+q$ is even, then $p_1$ and $q_1$ are even numbers. Let $p_1=2p_2$
and $q_1=2q_2$; then 
\begin{equation*}
\left( x,y\right) =\frac 1{p_2^2+q_2^2}\left(
r_1p_2+s_1q_2,s_1p_2-r_1q_2\right) \text{.}
\end{equation*}
These are included in the first and second types of general tiles.
\end{compactenum}

\smallskip\noindent\textbf{G1.} In this case
\begin{gather*}
\left( x,y\right) =\left( a,b\right) +\frac 1{p^2+q^2}\left(
rp+sq,sp-rq\right) \text{,}\\
\left( a,b\right) =\left( \frac 12+\frac 12\frac{-\left( r_0+s_0\right)
p+q\left( r_0-s_0\right) }{p^2+q^2},\frac 12\frac{\left( r_0-s_0\right)
p+q\left( r_0+s_0\right) }{p^2+q^2}\right)\text{.}
\end{gather*}
\begin{compactenum}[ a)]
\item If $p+q$ is even, these are included in the first type of general tiles.
\item If $p+q$ is odd,
\begin{compactenum}[(i)]
\item if $r_0+s_0$ is even, let 
\begin{gather*}
r_1=r-\frac{r_0+s_0}2\text{,}\quad s_1=s+\frac{r_0-s_0}2\text{,}\\
\left( x,y\right) =\left( \frac 12,0\right) +\frac 1{p^2+q^2}\left(
r_1p+s_1q,s_1p-r_1q\right) \text{.}
\end{gather*}
\item if $r_0+s_0$ is odd, then $r_0+s_0-p-q$ and $r_0-s_0-p+q$ are even. Let 
\begin{equation*}
r_1=\frac{r_0+s_0-p-q}2\text{,}\quad s_1=\frac{r_0-s_0-p+q}2\text{;}
\end{equation*}
then $r_0+s_0=2r_1+p+q$ and $r_0-s_0=2s_1+p-q$ and
\begin{equation*}
\left( a,b\right) =\left( \frac{-pr_1+qs_1}{p^2+q^2},
      \frac 12+\frac{ps_1+qr_1}{p^2+q^2}\right) \text{.}
\end{equation*}
Let $r_2=r-r_1$ and $s_2=s+s_1$;
\begin{equation*}
\left( x,y\right) =\left( 0,\frac 12\right) +\frac 1{p^2+q^2}\left(
r_2p+s_2q,s_2p-r_2q\right) \text{.}
\end{equation*}
\end{compactenum}
This case b) represents the third type of general tiles.
\end{compactenum}

\smallskip\noindent\textbf{H1.} In this case
\begin{gather*}
\left( x,y\right) =\left( a,b\right) +\frac 1{p^2+q^2}\left(
rp+sq,sp-rq\right) \text{,}\\
\left( a,b\right) =\left( 1-\frac{s_0p-r_0q}{p^2+q^2},
   \frac{r_0p+s_0q}{p^2+q^2}\right) \text{.}
\end{gather*}
\begin{compactenum}[ a)]
\item If $p+q$ is odd, these are included in the first type of general tiles.
\item If $p+q$ is even, let 
\begin{equation*}
r_1=r+p-s_0\text{,}\quad s_1=s+q+r_0\text{;}
\end{equation*}
\begin{equation*}
\left(x,y\right)=\frac 1{p^2+q^2}\left( r_1p+s_1q,s_1p-r_1q\right) \text{.}
\end{equation*}
These represent the second type of general tiles.
\end{compactenum}

\smallskip\noindent\textbf{H2.} In this case
\begin{gather*}
\left( x,y\right) =\left( a,b\right) +\frac 1{p^2+q^2}\left( r\left(
p+q\right) +s\left( q-p\right) ,s\left( p+q\right) +r\left( p-q\right)
\right) \text{,}\\ 
\left( a,b\right) =\left( 1-\frac{s_0\left( p+q\right) +r_0\left( p-q\right)}
                                 {p^2+q^2},
    \frac{r_0\left( p+q\right) +s_0\left( q-p\right) }{p^2+q^2}\right) \text{.}
\end{gather*}
Let
\begin{gather*}
r_1=r-s_0\text{,}\quad s_1=s+r_0\text{.}\\
\left( x,y\right) =\left( 1,0\right) +\frac 1{p^2+q^2}\left( r_1\left(
p+q\right) +s_1\left( q-p\right) ,s_1\left( p+q\right) +r_1\left( p-q\right)
\right) \text{.}
\end{gather*}
Let $p_1=p+q$ and $q_1=q-p$. Then $p=\frac{p_1-q_1}2$ and $q=\frac{p_1+q_1}2$
\begin{equation*}
\left( x,y\right) =\left( 1,0\right) +\frac 2{p_1^2+q_1^2}\left(
r_1p_1+s_1q_1,s_1p_1-r_1q_1\right) \text{.}
\end{equation*}
Notice that $p_1$ and $q_1$ are both even or both odd.
\begin{compactenum}[ a)]
\item If $p_1$ and $q_1$ are both odd, these are the fourth type of general
tiles.
\item If $p_1$ and $q_1$ are both even, let $p_1=2p_2$ and $q_1=2q_2$. Then
\begin{equation*}
\left( x,y\right) =\left( 1,0\right) +\frac 1{p_2^2+q_2^2}\left(
r_1p_2+s_1q_2,s_1p_2-r_1q_2\right) \text{.}
\end{equation*}
These are included in the first and second types of general tiles.
\end{compactenum}

\subsection{The exceptions}

Assume that the rotation center of order $4$ is in the middle of the edge of
one tile. Divide the tile in four identical squares $a_1$-$a_2$-$a_3$-$a_4$, 
$a_2$-$a_3$-$a_4$-$a_1$, $b_1$-$b_2$-$b_3$-$b_4$ and $c_1$-$c_2$-$c_3$-$c_4$, 
like it is shown in Figure \ref{fig:25}.

\begin{figure}[htb!]
  \centering
  \includegraphics[width=2.0323in]{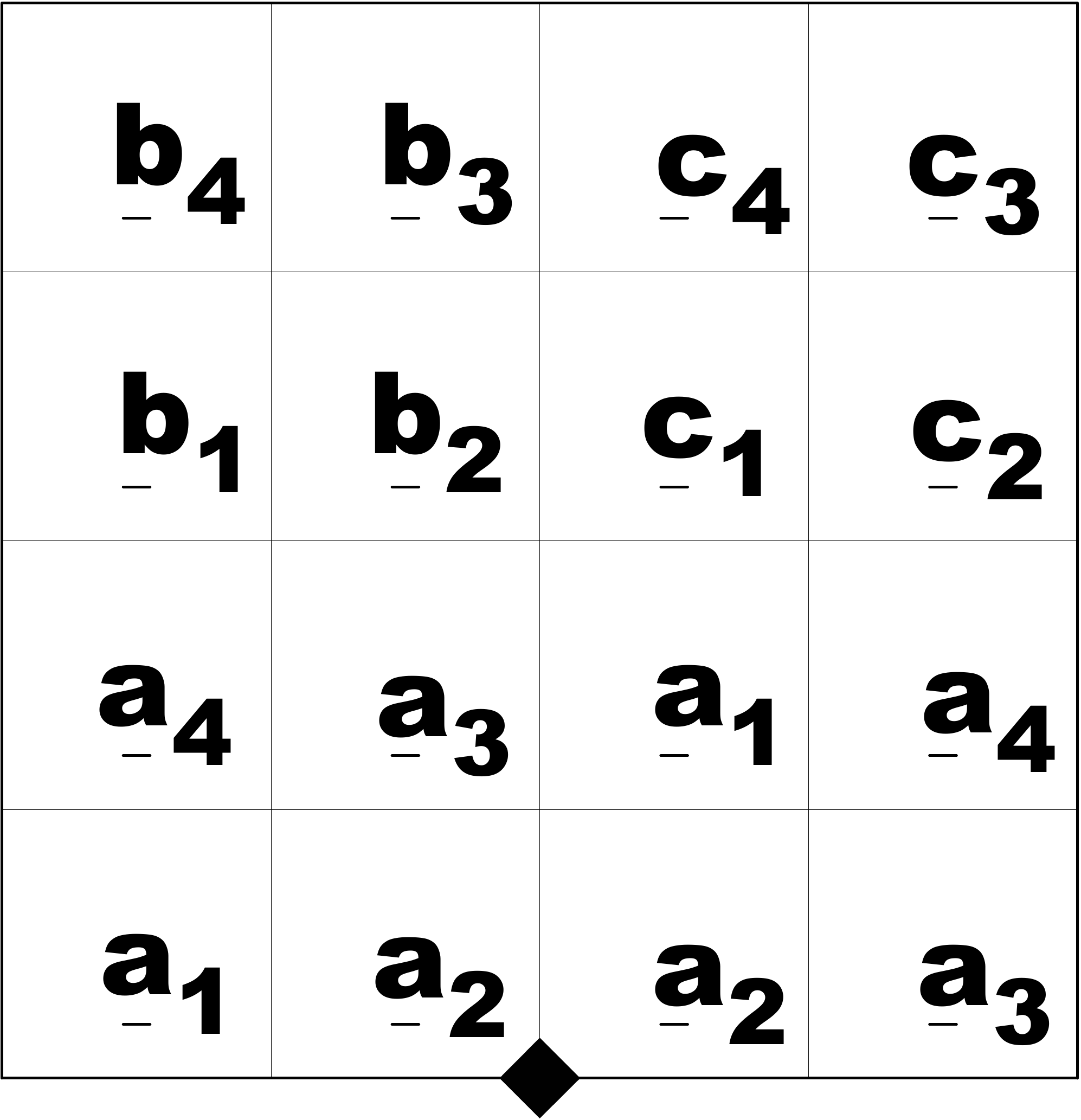}
  \vspace*{-8pt}
  \caption{}
  \label{fig:25}
\end{figure}

\begin{figure}[htb!]
  \centering
  \includegraphics[width=2.5659in]{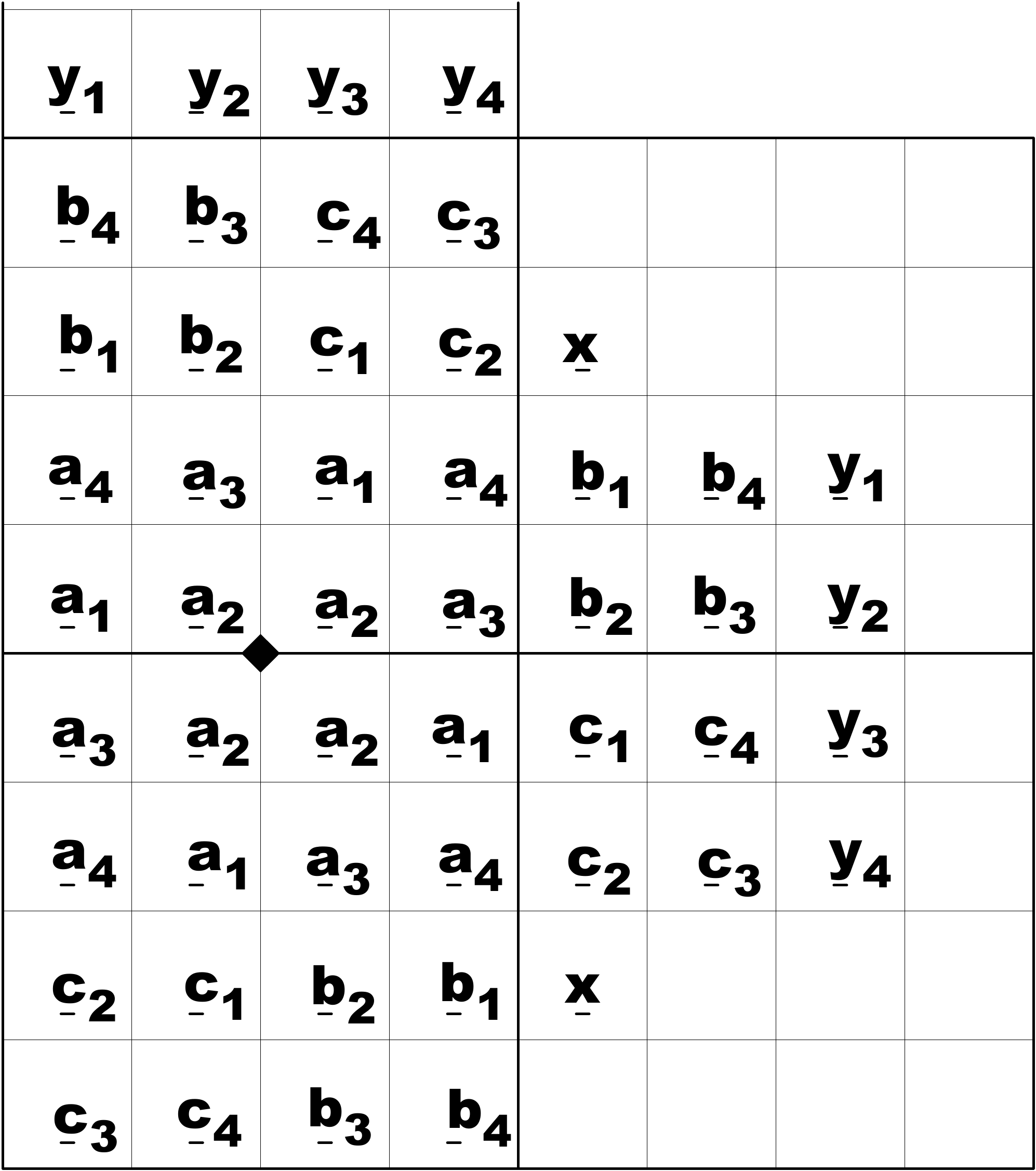}
  \vspace*{-4pt}
  \caption{}
  \label{fig:26}
\end{figure}

Figure \ref{fig:26} 
shows the rotation center of order $4$ in the middle of the common
edge of two tiles and two neighbor tiles. There are sixteen ways of choosing
these last tiles:
\begin{figure}[htb!]
  \centering
  \includegraphics[width=4.3414in]{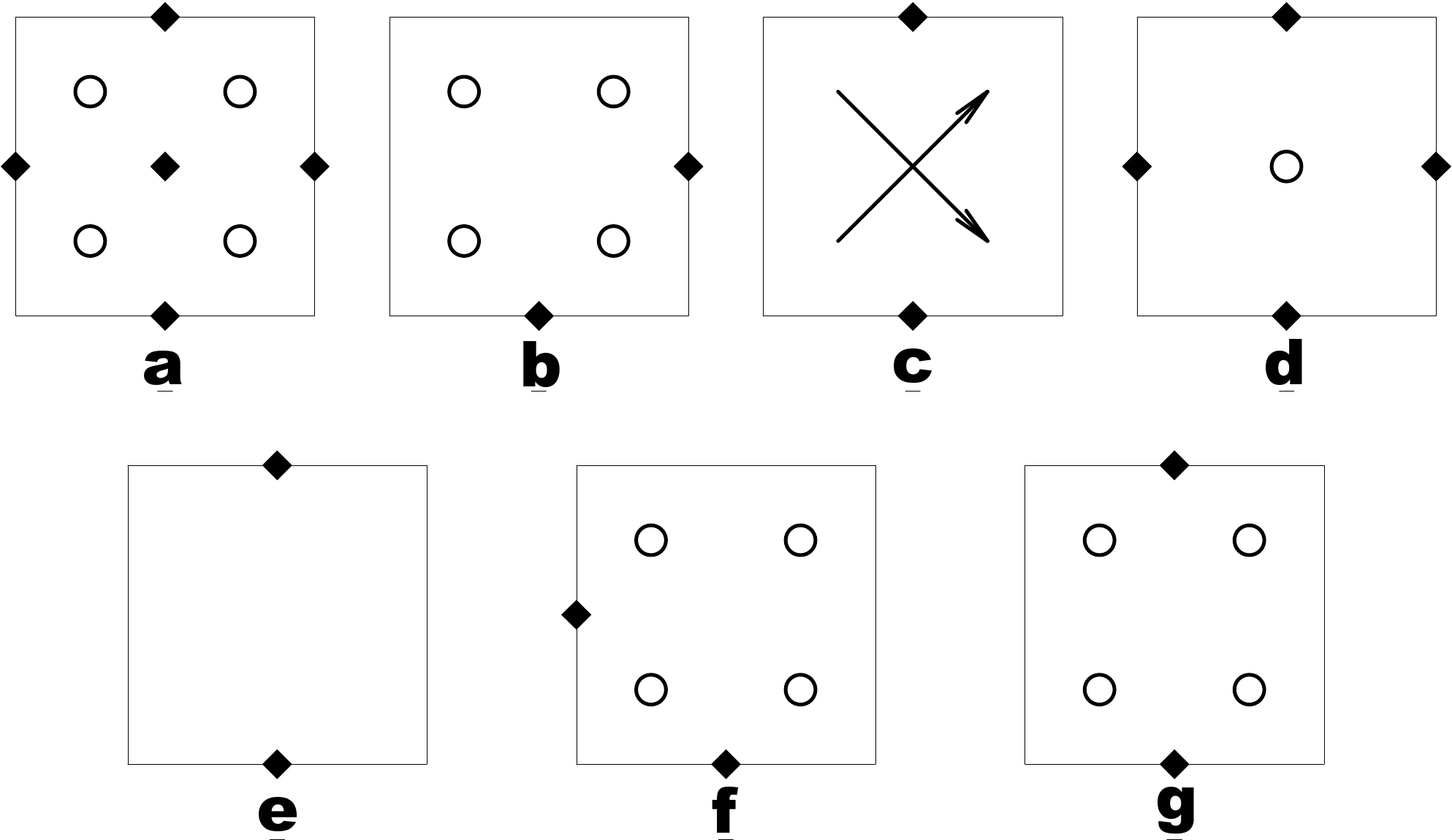}
  \vspace*{-8pt}
  \caption{}
  \label{fig:27}
\end{figure}
\begin{compactenum}[ a)]
\item In six of them 
\begin{align*}
a_1 &\equiv a_3\equiv b_2\equiv b_4\equiv c_1\equiv c_3\text{,}\\
a_2 &\equiv a_4\equiv b_1\equiv b_3\equiv c_2\equiv c_4\text{;}
\end{align*}
the tile is like it is shown in Figure \ref{fig:27}a;
\item in two of them 
\begin{align*}
a_1&\equiv a_3\equiv c_1\equiv c_3\text{,} \\
a_2&\equiv a_4\equiv c_2\equiv c_4\text{,} \\
b_1&\equiv b_3\text{,}\quad b_2\equiv b_4\text{;}
\end{align*}
the tile is like it is shown in Figure \ref{fig:27}b;
\item in one of them 
\begin{alignat*}{2}
a_1&\equiv b_4\equiv c_1\text{,}&\quad a_2&\equiv b_1\equiv c_2\text{,}\\
a_3&\equiv b_2\equiv c_3\text{,}&      a_4&\equiv b_3\equiv c_4\text{;}
\end{alignat*}
the tile is like it is shown in Figure \ref{fig:27}c;
\item in three of them 
\begin{alignat*}{2}
a_1&\equiv b_2\equiv c_3\text{,}&\quad a_2&\equiv b_3\equiv c_4\text{,}\\
a_3&\equiv b_4\equiv c_1\text{,}&      a_4&\equiv b_1\equiv c_2\text{;}
\end{alignat*}
the tile is like it is shown in Figure \ref{fig:27}d;
\item in one of them
\begin{equation*}
b_1\equiv c_2\text{,}\quad b_2\equiv c_3\text{,}\quad b_3\equiv c_4\text{,}\quad
b_4\equiv c_1\text{;}
\end{equation*}
the tile is like it is shown in Figure \ref{fig:27}e;
\item in two of them 
\begin{align*}
a_1&\equiv a_3\equiv b_2\equiv b_4\text{,}\\
a_2&\equiv a_4\equiv b_1\equiv b_3\text{,}\\
c_1&\equiv c_3\text{,}\quad c_2\equiv c_4\text{;}
\end{align*}
the tile is like it is shown in Figure \ref{fig:27}f;
\item in one of them 
\begin{gather*}
b_1\equiv b_3\equiv c_2\equiv c_4\text{,}\quad
b_2\equiv b_4\equiv c_1\equiv c_3\text{,}\\
y_1\equiv a_3\text{,}\quad
y_2\equiv y_3\equiv a_4\text{,}\quad
y_4\equiv a_1\text{;}
\end{gather*}
there are four possibilities for the $y$'s, but only one of them ($a_1\equiv
a_3$, $a_2\equiv a_4$) gives a different tile; it is like is shown in 
Figure \ref{fig:27}g.
\end{compactenum}

The tiles in Figure \ref{fig:27}a and \ref{fig:27}d can be obtained by
the general rule, making $p=2$ and $q=0$ in the first one, and $p=1$
and $q=1$ in the second one. The tile in Figure \ref{fig:27}e can be
obtained by the general rule, making $p=1$ and $q=0 $; it is of the
third type.

The tiles in Figure \ref{fig:27}b and \ref{fig:27}f are identical.

Hence, if we do not consider reflections, there are three exceptions to the
general rule, that are shown in Figures \ref{fig:13}--\ref{fig:15}.

Considering reflections, it is not difficult to see that there are the eight
possibilities represented in Figures \ref{fig:13}--\ref{fig:15}.

\newpage

\section*{Appendix: three patterns with Eduardo Nery tile}

\begin{figure}[H]
  \centering
  \includegraphics[width=5.7199in]{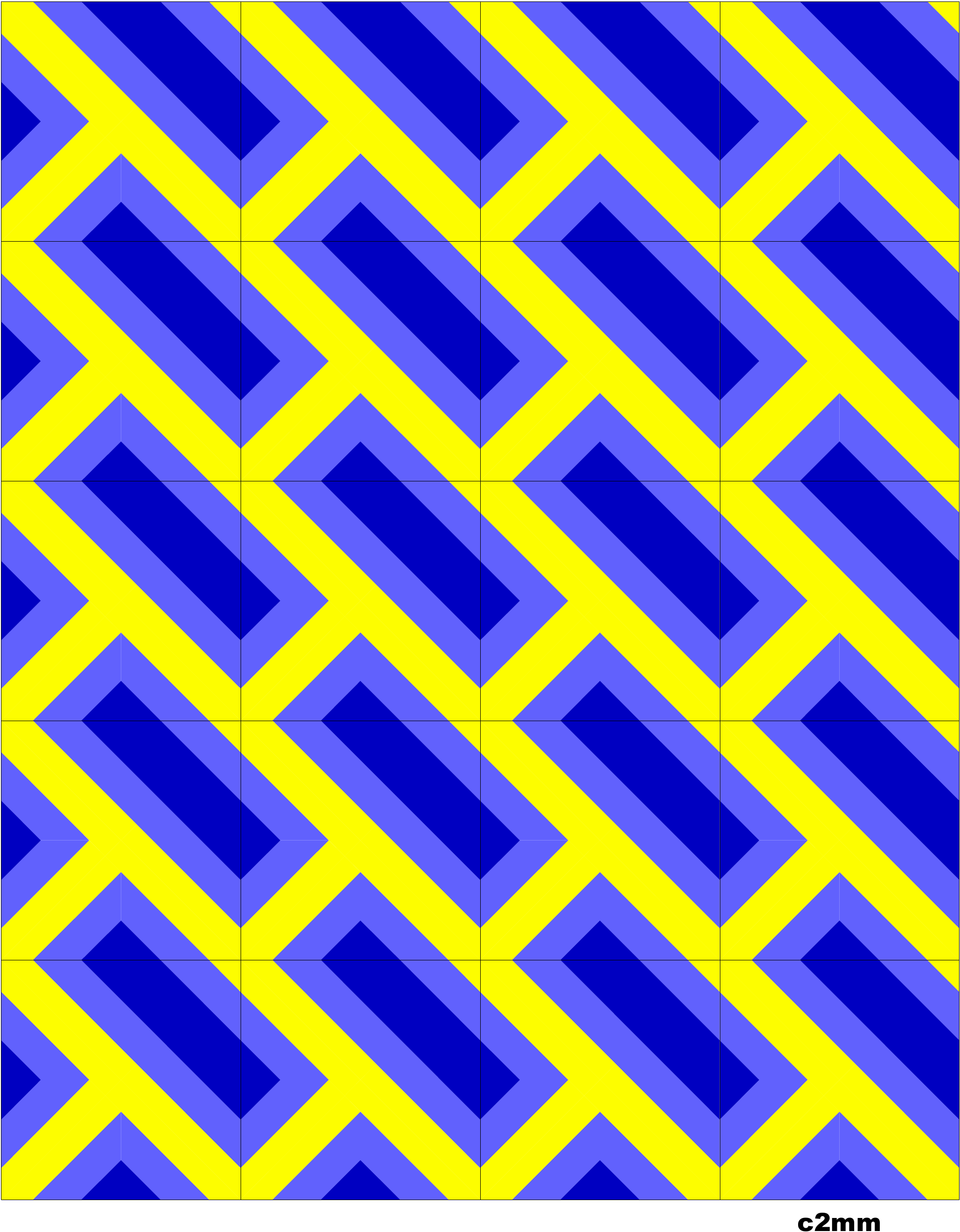}
  \vspace{-8pt}
  \caption{}
  \label{fig:28}
\end{figure}

\newpage

\begin{figure}[H]
  \centering
  \includegraphics[width=5.7199in]{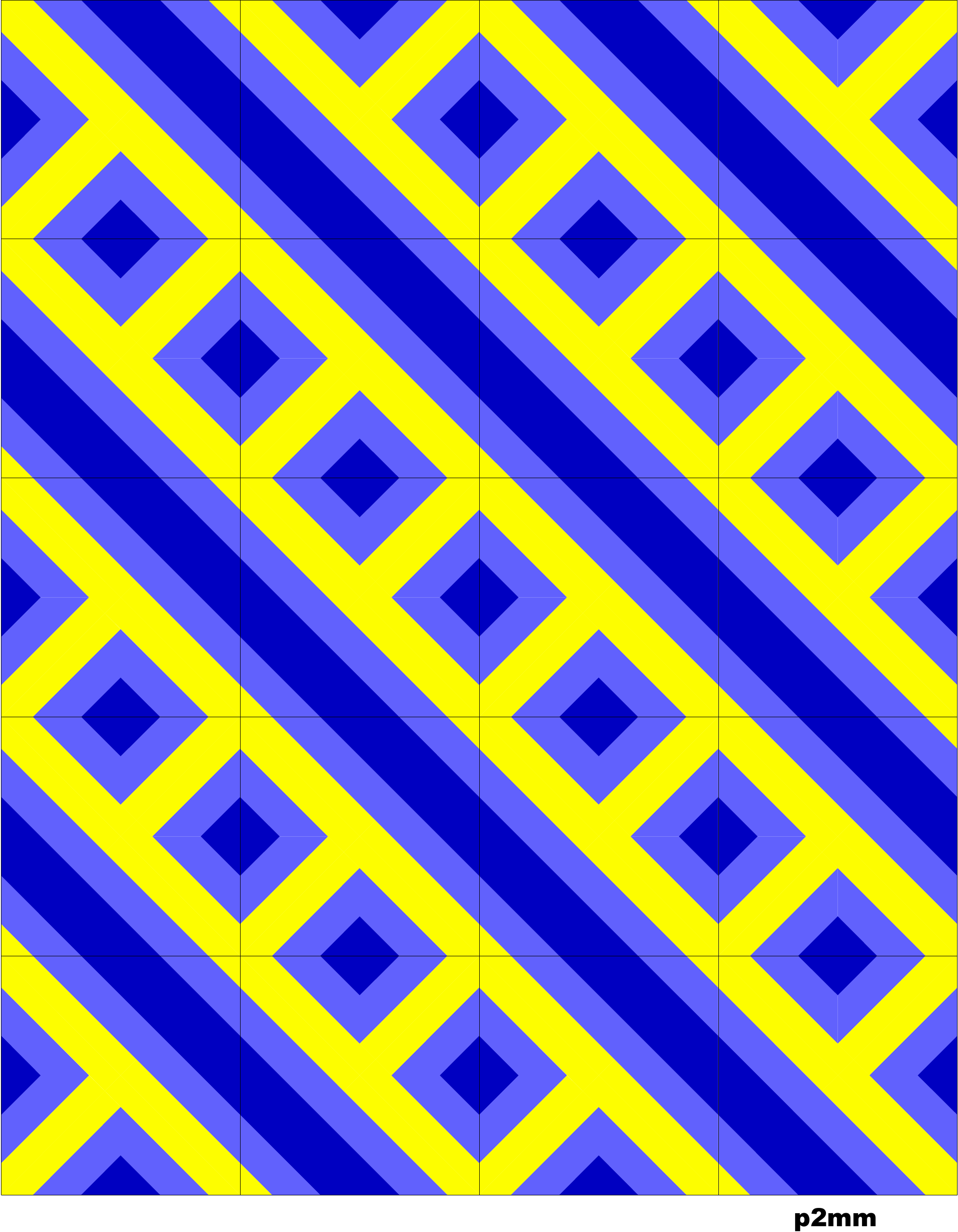}  
  \vspace{-8pt}
  \caption{}
  \label{fig:29}
\end{figure}

\newpage

\begin{figure}[H]
  \centering
  \includegraphics[width=5.7199in]{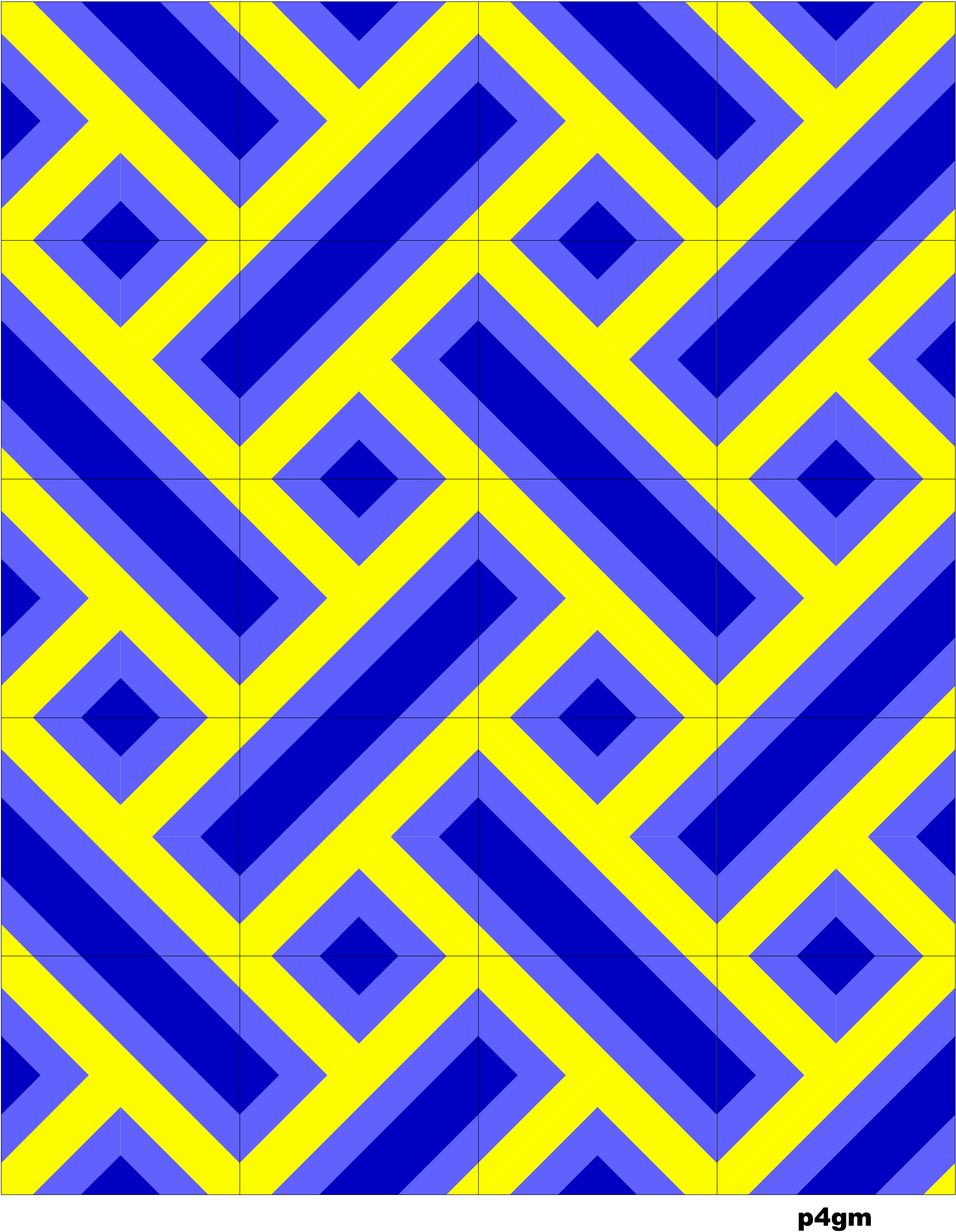}
  \vspace{-8pt}
  \caption{}
  \label{fig:30}
\end{figure}

\end{document}